\RequirePackage{fix-cm}
\documentclass[final,numbook]{svjour3}                     
\smartqed  
\usepackage{graphicx}
\usepackage{amsmath} 
\usepackage{amssymb}  
\usepackage{graphics}
\usepackage{epsfig}
\usepackage{framed}
\usepackage{cite}
\newtheorem{assumption}[property]{Assumption}

\newcommand{\beq}{\begin{equation}}
\newcommand{\eeq}{\end{equation}}
\newcommand{\beqa}{\begin{eqnarray}}
\newcommand{\eeqa}{\end{eqnarray}}
\newcommand{\beqs}{\begin{equation*}}
\newcommand{\eeqs}{\end{equation*}}
\newcommand{\beqas}{\begin{eqnarray*}}
\newcommand{\eeqas}{\end{eqnarray*}}

\newcommand{\R}{\mathbb R}

\newcommand{\hesslb}{c}
\newcommand{\hessub}{C}
\newcommand{\gammalb}{{\underline{\gamma}}}
\newcommand{\gammaub}{{\overline{\gamma}}}

\newcommand{\gradfjxk}{{\nabla f_j (x_j^k)}}
\newcommand{\hessfjxk}{{\nabla^2 f_j (x_j^k)}}
\newcommand{\gradfi}{{\nabla f_i}}

\newcommand{\bik}{{B_i^k}}

\newcommand{\hatk}{\hat{k}}

\newcommand{\hata}{\hat{A}}
\newcommand{\hatr}{\hat{\rho}}
\newcommand{\hatg}{\hat{\gamma}}
\newcommand{\p}{{\bar{\xi}}} 
 \journalname{Noname}
\begin{document}

\title{A globally convergent incremental Newton method
}

\titlerunning{Incremental Newton}        

\author{M. G\"urb\"uzbalaban \and \\ A. Ozdaglar \and \\ P.~Parrilo}


\institute{M. G\"urb\"uzbalaban \and A. Ozdaglar \and P. Parrilo \at
              Laboratory for Information and Decision Systems, \\ 
              Massachusetts Institute of Technology, Cambridge, MA, 02139.  \\
              \email{\{mertg, asuman, parrilo\}@mit.edu} 
 \\
             {* This work was supported by AFOSR MURI FA9550-09-1-0538 and ONR Basic Research Challenge No. N000141210997.                          }   
}

\date{Received: date / Accepted: date}

\maketitle

\begin{abstract}
Motivated by machine learning problems over large data sets and distributed optimization over networks,
we develop and analyze a new method called incremental Newton method for minimizing the sum of a large number of strongly convex functions. We show that our method is globally convergent for a variable stepsize rule. We further show that under a gradient growth condition, convergence rate is linear  for both  variable and constant stepsize rules. By means of an example, we show that without the gradient growth condition, incremental Newton method cannot achieve linear convergence. Our analysis can be extended to study other incremental methods: in particular, we obtain a linear convergence rate result for the incremental Gauss-Newton algorithm under a variable stepsize rule.
\end{abstract}

\section{Introduction}
\label{intro}
We consider the following unconstrained optimization problem where the objective function is the sum of component functions:
\beqa\label{pbm-multi-agent}
	\mbox{minimize  }&& f(x) = \sum_{i=1}^m f_i(x)  \\
	\mbox{subject to}&& x \in \R^n, \nonumber
\eeqa
where each $f_i:\R^n \to \R$ is a strongly convex and twice continuously differentiable function. This problem arises 
in many applications including least squares or more general parameter estimation problems (where $f_i(x)$ is the loss function representing the error between prediction of a parametric model obtained from data and the actual output), distributed optimization over networks (where $f_i(x)$ is the local objective function of agents connected through a network), dual formulation of problems with many constraints, and minimization of expected value of a function (where the expectation is taken over  a finite probability distribution or approximated by an $m$-sample average) (see e.g., \cite{bertsekas2011incremental,RamNedicVeer2007,Jordan13DistLearningApi,Boyd2011AdmmBook,Dickenstein2014,NedicOzdaglar09,Nedic2007rate,BottouLecun2005}). An important feature of these problems is that the number of component functions $f_i$ is large and not all simultaneously available. One is therefore interested in optimization algorithms that can iteratively update the estimate for an optimal solution using partial information about component functions.

One widely studied approach is the incremental gradient method, which cycles through the component functions using a deterministic order and updates the iterates using the gradient of a single component function. This method typically outperforms non-incremental methods in numerical studies since each inner iteration makes reasonable progress. However, it typically has sublinear convergence rate as it requires the stepsize to go to zero to obtain convergence to the optimal solution of problem (\ref{pbm-multi-agent}) (see \cite{bertsekas2011incremental}).

In this paper, we present an incremental Newton (IN) method that cycles deterministically through the component functions $f_i$ and uses the gradient of $f_i$ to determine the direction of motion and the Hessian of $f_i$ 
to construct the Hessian of the sum of component functions, $f$. Our main results can be summarized as follows:

First, we adopt a variable stepsize rule, which was introduced in Moriyama \textit{et al}. \cite{AlgEkfs2003} for the analysis of the incremental Gauss-Newton method with an adaptive stepsize rule. The stepsize measures the progress of the iterates over a cycle relative to the progress in the inner iterations and aims to dampen the oscillations associated with incremental methods in the ``region of confusion" (i.e., the set over which the component functions have non-aligned gradients; see e.g. \cite[Example 1.5.5]{Bertsekas99nonlinear}). We show that our IN algorithm is globally convergent with this variable stepsize rule. 

Second, we adopt an assumption, which we refer to as the {\it gradient growth} assumption, which states that norms of gradients of $f_i$'s are bounded from above by a linear function of the norm of $f$. Under this assumption we show that the normalized stepsize sequence (normalization of stepsize by the iteration number $k$ is used since the Hessians are accumulated at each step) remains bounded away from zero and provide an explicit characterization of this bound in terms of problem parameters. Our analysis relies on viewing the IN method as an inexact perturbed Newton method. We use the lower and upper bounds on the stepsize sequence together with bounds on the Hessian of iterates to provide bounds on the Hessian error and the gradient error of the method. This allows us to use the convergence rate results on inexact perturbed Newton methods to show that 
IN method converges locally linearly to the optimal solution of problem (\ref{pbm-multi-agent}). Under some additional assumptions, we show that IN method achieves asymptotically error-free curvature (or Hessian matrix of $f$) estimates which do not extend to many incremental quasi-Newton methods (see Remark \ref{rema-hessian-accurate-at-limit}). However, our global convergence and linear convergence rate results admit extensions to incremental quasi-Newton methods. Our analysis can also be extended to study \textit{incremental Gauss-Newton} method under a variable stepsize rule for solving least square problems, also known as the extended Kalman filter (EKF) method with variable stepsize or equivalently the EKF-S algorithm \cite{AlgEkfs2003}, and shows linear convergence rate for this method, thus answering a problem left open in Moriyama \textit{et al}. \cite[\S 7]{AlgEkfs2003}. Note that the incremental Gauss-Newton method without the variable stepsize shows typically sublinear convergence behavior \cite{Bertsekas1996incremental, Davidon76, Bertsekas99nonlinear}.

Third, we show that under the gradient growth condition, the IN method converges globally linearly to the optimal solution for a sufficiently small constant stepsize. The analysis of our algorithm under a constant stepsize rule uses bounds on the gradient errors, which may be of independent interest. 

Fourth, we provide an example that shows that without the gradient growth condition, IN method cannot converge faster than sublinear, thus highlighting the importance of gradient growth condition in the performance of the IN method.

Our work is related to the literature on incremental gradient (IG) methods (see \cite{Solodov98IncrGrad,Bertsekas99nonlinear,bertsekas2011incremental,BersekasIncrGrad97}). The randomized version of the IG method, also referred to as the {\it stochastic gradient descent} (SGD) \cite{RobbinsMonro,Schraudolph2007stochastic,BottouLecun2005}, has been popular and used extensively for solving machine learning problems \cite{Borges2010BookChapter,BottouLecun2005,Zhang2004}. Many variants of the IG method are proposed to accelerate its convergence, including the \textit{IG method with momentum} of Tseng  \cite{TsengIncrGradient98} and Mangasarian \textit{et al}. \cite{ManSol98IncGradMomentum}. Tseng's approach with momentum \cite{TsengIncrGradient98} requires once in a while constructing the gradient of $f$, and can be hard to implement in problems where the entire objective function is not available. Another interesting class of methods includes the  \textit{incremental aggregated gradient} (IAG) method of Blatt {\it et al.} \ (see \cite{Blatt2007incremental,TsengYun2014incremental}) and closely-related stochastic methods including the \textit{stochastic average gradient} (SAG) method \cite{Leroux2012sgd}, the SAGA method \cite{BachSagaMethod14} and the MISO method \cite{MairalSurrogate2013}. These methods process a single component function at a time as in incremental methods, but keeps a memory of the most recent gradients of all component functions so that a full gradient step 
is taken at each iteration. They have been shown to have fast convergence properties but may require an excessive amount of memory when $m$ is large.

There has also been a recent interest in incremental and stochastic quasi-Newton methods, motivated by numerical evidence showing that second-order methods are faster than first-order methods in many practical problems \cite{StocBFGSRibeiro14,StocBFGSNocedal14,Bordes2009sgdqn,Dickenstein2014}. 
In particular, Mokhtari \textit{et al.} propose a stochastic BFGS algorithm with a $O(1/k)$ convergence result \cite{StocBFGSRibeiro14}. Byrd \textit{et al.} \cite{StocBFGSNocedal14} develop a stochastic quasi-Newton algorithm that avoids the potentially harmful effects of differencing stochastic gradients that can be noisy, although no convergence analysis is given. SGD-QN algorithm \cite{Bordes2009sgdqn}, AdaGrad algorithm \cite{Duchi2011AdaGrad}, oBFGS and oLBFGS algorithms~\cite{Schraudolph2007stochastic}, SFO algorithm \cite{Dickenstein2014} are among other recent second-order stochastic methods that use quasi-Newton approaches. DANE algorithm \cite{SrebroDane2013} is a Newton-like method based on mirror-descent type updates with a linear convergence rate when the functions $f_i$ are quadratics, although to the best of our knowledge no convergence rate results are currently known beyond quadratic objective functions.
 
\paragraph{Outline.} In Section \ref{sec:incr-gauss-newton}, we motivate and introduce the IN method deriving key lemmas for its analysis. In Section \ref{sec-conv-variable-stepsize}, first we show its global convergence under a variable stepsize rule. Then, we introduce the gradient growth assumption and under this assumption we prove local linear convergence. We also discuss implications of our analysis to the incremental quasi-Newton methods and the EKF-S algorithm. In Section \ref{sec-lin-conv-constant-step}, we start with deriving upper bounds on the norm of the gradient error in our method for an arbitrary stepsize and then show global linear convergence with a constant stepsize under the gradient growth assumption. In Section \ref{sec-examples-sublinear-conv}, we give examples illustrating the possible sublinear convergence of the IN method in case this assumption does not hold. We conclude by Section \ref{sec-discussion-and-future} with a summary of our results.
\section{The IN Method}
\label{sec:incr-gauss-newton}
Newton's method is an important classical method for solving smooth unconstained optimization problems of the form   
  \beqas\mbox{minimize  }&& f(x)  \\
	\mbox{subject to}&& x \in \R^n. \nonumber 
   \eeqas
The standard Newton iteration is
  \beq\label{iter-standard-newton} x^{k+1} = x^k - \big(\nabla^2 f(x^k)\big)^{-1} \nabla f(x^k),
  \eeq
where $\nabla f(x)$ and $\nabla^2 f(x)$ denote the gradient and Hessian of $f$ at $x \in \R^n$ respectively. In this paper, we focus on problem \eqref{pbm-multi-agent} where  
  $$ f(x) = \sum_{i=1}^m f_i(x).$$
Computing $\nabla f(x^k)$ and $\nabla^2 f(x^k)$ in \eqref{iter-standard-newton}
necessitates summing gradients and Hessian matrices of each of the component functions $f_i$ which may be costly when the number of the functions $m$ is large. For such problems, it may be more effective to use an \textit{incremental method}, which cycles through each of the component functions $f_i$ and update $x^k$ based on the gradient $\nabla f_i$ and the Hessian $\nabla^2 f_i$ (see e.g. \cite{Bertsekas99nonlinear}).  
Our aim is to provide an incremental version of the Newton method. 
A natural idea would be to approximate $\nabla f$ with $\nabla f_i$, where $i \in \{1,2,\dots,m\}$ varies in a cyclic manner and to construct $\nabla^2 f$ by incrementally summing $\nabla^2 f_i$'s. These observations motivate the following algorithm, which we call the \textit{incremental Newton} (IN) algorithm. Given an initial point $x_1^1 \in \R^n$ and $m>1$, we consider the iterations   
      \beqa\label{eq-inner-update} \mbox{(IN)} \quad \quad x_{i+1}^k &:=& x_i^k - \alpha^k \big( H_i^k \big)^{-1}  \nabla f_i(x_i^k), \quad i = 1, 2, \dots,m, \\
	x_1^{k+1} &:=& x_{m+1}^k, \label{eq-inner-update-convention}
	\eeqa
where $H_i^k$ is a symmetric matrix updated by 
\beqa \label{eq-hessian-update} H_i^k &:=&  H_{i-1}^k + \nabla^2 f_i(x_i^k), \quad i = 1, 2, \dots,m,\\
     H_0^{k+1} &:=& H_{m}^k, \quad H_0^1 := 0, \label{eq-hessian-update-convention}
\eeqa
and $\alpha^k>0$ is the stepsize. The matrices $H_i^k$ accumulate and capture the second-order information at the iterates. For a fixed value of $k\geq 1$ and $i \in \{1,2,\dots,m\}$, we refer to the update \eqref{eq-inner-update} as an \textit{inner iteration}.  Consecutive $m$ iterations starting with $i=1$  will be denoted as a \textit{cycle} of our algorithm.

Algorithm IN is reminiscent of the EKF algorithm (when $\alpha^k=1$)  \cite{Bertsekas1996incremental} or the EKF algorithm with variable stepsize (EKF-S algorithm) \cite{AlgEkfs2003}, but there are major differences: EKF and EKF-S are Gauss-Newton based methods designed specifically for the least square problems using only first-order derivatives whereas Algorithm IN applies not only to least square problems, but also to problem \eqref{pbm-multi-agent} and is a Newton-based method that uses second-order derivatives in addition to first-order derivatives. When $\alpha^k=1$, it can be shown that the IN iterations satisfy
	\begin{eqnarray} \label{iter-incr-gauss-3} 
x_{i+1}^k &=& \arg \min_{x \in \R^n} \sum_{j=1}^i \hat{f}_j(x, x_j^k), \quad i=1,2,\dots,m,
	\end{eqnarray}
where $\hat{f}_j(x,x_j^k)$ is the standard quadratic approximation to $f_j$ around the point $x_j^k$ formed by the Taylor's series expansion given by 
\beq\label{eq-quadratic-approx-to-f}
 \hat{f}_j(x, x_j^k) = f_j(x_j^k) + \nabla f_j(x_j^k)^T (x - x_j^k) 
   + \frac{1}{2}(x - x_j^k) ^T \nabla^2 f_j(x_j^k) (x - x_j^k). 
\eeq
Thus, when each function $f_j$ is a quadratic, we have $\hat{f}_j = f_j$ for $j=1,2,\dots,m$ and it suffices to have only one cycle ($m$ inner iterations) of the IN method to reach out to the globally optimal solution. This is clearly much faster than first-order methods or the Newton-like methods such as the DANE method \cite{SrebroDane2013} which has only linear convergence for quadratic $f_i$'s. However, the trade-off for this accuracy in our method is increased memory requirement $O(n\times n)$ and the additional computation of the second-order derivatives.
   
We start with a lemma that provides a characterization for the evolution of inner iterations.  
\begin{lemma}\label{lemm:conseq-iters}Let $\{x_1^k, x_2^k,\dots, x_m^k\}$ be the iterates formed by the IN algorithm given by \eqref{eq-inner-update}--\eqref{eq-hessian-update-convention}. Then, for $i=1,2,\dots,m$, we have 
\beq \label{eq:conseq-outer-loop}
  x_{i+1}^k = x_1^k - \alpha^k (H_i^k)^{-1} \sum_{j=1}^i \bigg(\nabla f_j (x_j^k) + \frac{1}{\alpha^k} \nabla^2 f_j(x_j^k)(x_1^k - x_j^k) \bigg).
\eeq 
\end{lemma}
\begin{proof}
Let $x_1^k$ be given. The iterations \eqref{eq-inner-update} can be rewritten as
$$x_{i+1}^k = x_i^k - \big(H_i^k\big)^{-1} \big( \alpha^k  \nabla  f_i (x_i^k)\big), \quad i=1,2,\dots,m, $$
which is equivalent to
\beq\label{eq-kalman-filter-update} x_{i+1}^k =  x_i^k + \big(H_i^k\big)^{-1} \big(C_i^k\big)^T ( z_i^k - C_i^k x_i^k), \quad i=1,2,\cdots,m,
\eeq
where $C_i^k$ is a positive definite matrix satisfying
$$ \quad (C_i^k)^T C_i^k = \nabla^2 f_i(x_i^k), \quad H_{i}^k = H_{i-1}^k + (C_i^k)^T C_i^k, $$  
$$z_i^k = - \big(C_i^k\big)^{-T} \big( \alpha^k \nabla f_i(x_i^k) -  \nabla^2 f_i(x_i^k) x_i^k \big).$$
(Such a matrix $C_i^k$ exists and can for instance be obtained by a Cholesky decomposition of the positive definite matrix $\nabla^2 f_i(x_i^k)$). Then, the update formula \eqref{eq-kalman-filter-update} is equivalent to a Kalman filter update that solves the incremental quadratic optimization problem 
     \beqas x_{i+1}^k &=& \arg \min_{x \in \R^n} \sum_{j=1}^i \| z_j^k - C_j^k x \|^2, \quad \mbox{for} \quad i=1,2,\cdots,m,  
     \eeqas 
(see \cite[Proposition 1.5.2]{Bertsekas99nonlinear}). Using existing results on Kalman filters, (in particular \cite[Proposition 1.5.2]{Bertsekas99nonlinear}), we obtain
  \beqas x_{i+1}^k &=& x_1^k + \big(H_i^k\big)^{-1} \sum_{j=1}^i (C_j^k)^T (z_j^k - C_j^k x_1^k) \\
      &=& x_1^k - \big(H_i^k\big)^{-1} \bigg( \sum_{j=1}^i  \alpha^k \nabla f_j(x_j^k) + \nabla^2 f_j(x_j^k) (x_1^k - x_j^k) \bigg)
  \eeqas
which is equivalent to \eqref{eq:conseq-outer-loop}. This completes the proof. \qed
\end{proof}

Using Lemma \ref{lemm:conseq-iters}, for $i=1,2,\dots,m$, we can write
\beq\label{eq-iter-ver-compact} x_{i+1}^{k} = x_1^k - \alpha^k D_i^k h_i^k 
\eeq
where 
\beq\label{eq-iter-ver-compact2}
D_i^k = (H_i^k)^{-1}, \quad h_i^k = \sum_{j=1}^i \bigg( \gradfjxk + \frac{1}{\alpha^k} \hessfjxk (x_1^k - x_j^k) \bigg).
\eeq

We make the following two assumptions which have been used in a number of papers for analyzing incremental and stochastic methods (see e.g. \cite{Bertsekas1996incremental}, \cite{AlgEkfs2003},~\cite[Theorem 3.1]{ManSol98IncGradMomentum},~\cite[Assumption 1]{StocBFGSRibeiro14},~\cite{SrebroDane2013}, \cite{BachSagaMethod14}).

\begin{assumption}\label{assum-compact} \textbf{(Boundedness)} The sequence $\{x_1^k, x_2^k, \dots, x_m^k \}_{k=1,2,\dots}$ generated by the IN iterations \eqref{eq-inner-update}-\eqref{eq-hessian-update-convention} is contained in a compact set $\mathcal{X} \in \R^n$ whose diameter is
   \beq\label{def-R-diameter} R:=\max_{x,y \in \mathcal{X}} \|x -y\|.
   \eeq 
\end{assumption}
\begin{assumption}\label{assum-strong-cvx} \textbf{(Hessian boundedness)} The functions $f_i$, $i=1,2,\dots,m$ are twice continuously differentiable, and there exists constants $c>0$ and $C>0$ such that\footnote{Note that the existence of the upper bound $C$ on the Hessian is an immediate implication of Assumption \ref{assum-compact} as $f$ is twice continuously differentiable. We include it here to highlight the Hessian bounds in the same place.} 
	\beq\label{eq-hessian-bounds}  c I \preceq 		\nabla^2 f_i (x)  \preceq C I, 
 	\eeq 
for all $x \in \R^n$ and $i=1,2,\dots,m$.
\end{assumption}
A consequence of Assumption \ref{assum-strong-cvx} is that the function $f$ is strongly convex with parameter $cm>0$ as each $f_i$ is strongly convex with parameter $c>0$. Thus, the optimization problem \eqref{pbm-multi-agent} admits a unique optimal solution. Another consequence is that the gradients have a Lipschitz constant $\hessub$, i.e., 
  \beq\label{eq:gradient-Lipschitzness}\|f_i(x) - f_i(y)\| \leq \hessub \|x-y\|, \quad i=1,2,\dots,m,
  \eeq
for all $x,y \in \R^n$, where we use $\|\cdot\|$ to denote the $2$-norm (Euclidean norm) of a vector or the 2-norm (spectral norm) of a matrix depending on the context throughout this paper. We note that, by \eqref{eq-hessian-bounds}, the ratio
   \beq\label{def-cond-number} Q:= \frac{C}{c}
   \eeq
is an upper bound for the condition number of the Hessian matrices at the iterates. Therefore, we will refer to it as the \textit{condition number} of problem \eqref{pbm-multi-agent}.
 
We now investigate the evolution of the Hessian matrices $H_i^k$ with $k$. It is straightforward to see from the Hessian update formula \eqref{eq-hessian-update} that
\beq\label{eq-hessian-iterates-decompositon} H_i^k = \sum_{i=1}^{k-1} \sum_{j=1}^m \nabla^2 f_j (x_j^i) + \sum_{j=1}^i \hessfjxk. 
\eeq
The next lemma shows that the matrices $\{H_i^k\}_{i,k\geq1}$ have a norm (2-norm) growing linearly with $k$. 

\begin{lemma}\label{eq-iter-hess-bound} Suppose that Assumptions \ref{assum-compact} and \ref{assum-strong-cvx} hold. Then, for any $i = 1,2, \dots, m$ and $k\geq 1$, we have 
\beq\label{eq-hessian-growth-inner-steps} ck I \preceq c\big( (k-1)m + i\big) I \preceq   H_i^k \preceq Cmk I,  
\eeq
\beq\label{eq-hessian-inv-growth-inner-steps} \frac{1}{Cmk} I \preceq   D_i^k  \preceq \frac{1}{c\big( (k-1)m + i\big)} I \preceq \frac{1}{ck} I.
\eeq
It follows that, for any $i=1,2,\dots,m$ and $k\geq 2$,
  \beq\label{ineq-hessian-growth-k-geq-2}\frac{ckm}{2}I \preceq H_i^k, \quad D_i^k \preceq \frac{2}{ckm}I.
  \eeq  
\end{lemma}
\begin{proof}
The first inequality on the left of \eqref{eq-hessian-growth-inner-steps} is a direct consequence of the inequality $k \leq \big( (k-1)m + i\big)$ for $k\geq 1$ and $i \in \{1,2,\dots,m\}$. Other inequalities in \eqref{eq-hessian-growth-inner-steps} and \eqref{eq-hessian-inv-growth-inner-steps} follow directly from applying the Hessian bounds \eqref{eq-hessian-bounds} to the representation of $H_i^k$ given by the formula \eqref{eq-hessian-iterates-decompositon} and the fact that $D_i^k = (H_i^k)^{-1}$. Inequalities \eqref{ineq-hessian-growth-k-geq-2} follow from \eqref{eq-hessian-growth-inner-steps} and \eqref{eq-hessian-inv-growth-inner-steps} using the fact that $k/(k-1)\leq 2$ for $k\geq 2$.
\qed
\end{proof}
\section{Convergence with variable stepsize}
\label{sec-conv-variable-stepsize}
In this section, we introduce a variable stepsize rule and  study global convergence of the IN method under this stepsize. We use the variable stepsize that was proposed in Moriyama \textit{et al.} \cite{AlgEkfs2003} for showing convergence of the EKF-S method. 
\begin{assumption}\label{assum-stepsize} \textbf{(Variable stepsize)} The stepsize used in the Algorithm IN defined by \eqref{eq-inner-update}--\eqref{eq-hessian-update-convention} satisfies 
$$ 1 \leq \alpha^k \leq \max(1,\alpha_*^k), \quad k=1,2,\dots, $$
where
\large
\beq\label{def-alpha-star-k}  \alpha_*^k = 
\begin{cases} 
      \frac{1-\eta}{\hessub} \frac{ (x_1^{k+1} - x_1^k)^T H_m^k (x_1^{k+1} - x_1^k)}{\|x_1^{k+1} - x_1^k\| \sum_{i=2}^m \|x_i^{k} - x_1^k\| + \frac{m}{2} \|x_1^{k+1} - x_1^k\|^2},   \hfill & \text{ if $x_1^k \neq x_1^{k+1}$, } \\
      0, \hfill & \text{ otherwise},
  \end{cases}
\eeq
\normalsize  
for some  $ \eta \in (0,1)$, which we refer to as the \emph{stepsize control parameter}.
\end{assumption}
The form of this stepsize can be motivated as follows: The representation formulae for the inner iterates \eqref{eq-iter-ver-compact}--\eqref{eq-iter-ver-compact2} show that if the norm of $\nabla f$ is very small compared to the norm of $\nabla f_i$ for some $i$, unless the stepsize $\alpha^k$ is small, we could be in a situation where the total distance traveled during one cycle of the iteration $\|x_1^{k+1} - x_1^{k}\|$ is very small compared to that of the inner iterations $\|x_i^k - x_1^{k}\|$, resulting in large oscillations. 
This suggests to have a variable stepsize that takes smaller steps if the ratio $\|x_1^{k+1} - x_1^{k}\|/\sum_{i}\|x_i^k - x_1^{k}\|$ gets smaller as in Assumption \ref{assum-stepsize}. Such a variable stepsize would kill undesired oscillations, enabling moving towards the optimal solution in a more efficient way. This stepsize can also be motivated by a descent condition, which leads to the monotonic decrease of the function values $\{f(x_1^k)\}_k$ asymptotically (when $k$ is large enough)  as long as the iterates stay bounded (see \cite[Lemma 4.1]{AlgEkfs2003}).  Furthermore, it is easy to implement by a simple adaptive stepsize algorithm (see Remark \ref{rema-adaptive-stepsize}). 

By Assumptions \ref{assum-compact} and \ref{assum-strong-cvx} on the boundedness of the iterates, gradients and Hessian matrices, we see that $h_m^k$ defined by \eqref{eq-iter-ver-compact2} is bounded. Hence, by Lemma \ref{eq-iter-hess-bound} that provides bounds for the matrices $D_m^k$ and \eqref{eq-iter-ver-compact} on the evolution of inner iterates,
it follows that the distance between the consecutive iterates at the beginning of each cycle satisfies 
\beqas
	\|x_{1}^{k+1} - x_1^k\| = \alpha^k \|D_m^k h_m^k\| = O(\alpha^k/k).
\eeqas 
Thus, the \textit{normalized stepsize}
  \beq\label{eq-normalized-step}\gamma^k=\alpha^k/k
  \eeq
can be thought as the effective stepsize whose behavior determines the convergence rate. If $\alpha^k$ is bounded, then $\gamma^k \to 0$ in which case we would expect sublinear convergence in general as Example \ref{example-sublinear-conv} in Section \ref{sec-examples-sublinear-conv} shows. For faster convergence, we would need $\alpha^k$ and (hence $\alpha_*^k$ by Assumption \ref{assum-stepsize}) to grow with $k$. This motivates us to define
  \beq\label{def-limiting-stepsizes} \gamma_*^k=\alpha_*^k/k, \quad \underline{\gamma} = \liminf_{k\to \infty} \gamma_*^k \quad \mbox{and} \quad \overline{\gamma} = \limsup_{k\to \infty} \gamma_*^k, 
  \eeq
requiring a lower bound on the growth rate $\underline{\gamma}$. For linear convergence, we would also typically need an upper bound on the stepsize, because even the simplest methods (such as the steepest descent method) with constant stepsize require the stepsize to be small enough in order to be able to guarantee linear convergence \cite[Section 1.4.2]{Pol87}. This motivates the next result which provides an upper bound for $\overline{\gamma}$.
%
\begin{lemma}\label{lemm-alphastar-bound}
Suppose that Assumptions \ref{assum-compact}, \ref{assum-strong-cvx} and \ref{assum-stepsize} hold. Then, we have 
	\beq\label{ineq-gamma-star-k} \overline{\gamma} \leq \phi \quad \mbox{and} \quad \gamma_*^k \leq \phi \quad \mbox{for all} \quad k=1,2,\dots,
	\eeq
with \beq\label{def-phi}\phi = 2(1-\eta) Q > 0\eeq where $\eta$ is the stepsize control parameter as in \eqref{def-alpha-star-k} and $Q$ is the condition number of the problem \eqref{pbm-multi-agent} defined by \eqref{def-cond-number}.	
\end{lemma}
\begin{proof} 
If $h_m^k \neq 0$, then for $k\geq 2$, 
\begin{eqnarray*}
  \alpha_*^k &=& \frac{1-\eta}{\hessub} \frac{ (h_m^k)^T D_m^k h_m^k}{\|D_m^k h_m^k\| \sum_{i=2}^m \|D_{i-1}^k h_{i-1}^k\| + \frac{m}{2} \|D_m^k h_m^k\|^2}  \\
  &\leq& \frac{1-\eta}{\hessub} \frac{\frac{1}{cmk} \|h_m^k\|^2}{(1/Cmk)^2 (\| h_m^k\| \sum_{i=2}^m \| h_{i-1}^k\| + \frac{m}{2} \| h_m^k\|^2) }  \\
  &=& \phi  \frac{1}{ (2/m)\sum_{i=2}^m \| h_{i-1}^k\|/\| h_m^k\| + 1} k \\
  &\leq& \phi  k
\end{eqnarray*}
where the first equality follows from \eqref{eq-iter-ver-compact}, the first inequality is obtained by using the bound on $D_i^k$ for $i=1,2,\dots,m$ in Lemma \ref{eq-iter-hess-bound} and the second inequality follows since the term $(2/m)\sum_{i=2}^m \| h_{i-1}^k\|/\| h_m^k\|$ is non-negative. This implies \eqref{ineq-gamma-star-k}. Otherwise, if $h_m^k = 0$, then $\alpha_*^k=\gamma_*^k = 0$ satisfying \eqref{ineq-gamma-star-k} clearly. \qed
\end{proof}

The next theorem shows the global convergence of the iterates generated by the IN method to the unique optimal solution of problem \eqref{pbm-multi-agent}. The proof uses a similar line of argument as in the proof of 
\cite[Theorem 4.1]{AlgEkfs2003}; so we skip it here due to space considerations. 
\begin{theorem}\textbf{(Global convergence with variable stepsize)}\label{thm-global-conv}
Suppose that Assumptions \ref{assum-compact}, \ref{assum-strong-cvx} and \ref{assum-stepsize} hold. Then the iterates $\{x_1^k\}_{k=1}^\infty$ generated by the IN method \eqref{eq-inner-update}--\eqref{eq-hessian-update-convention} satisfy 
$$ \lim_{k \to \infty} \| \nabla f (x_1^k)\| = 0 $$
and converge to the unique optimal solution of the optimization problem \eqref{pbm-multi-agent}. 
\end{theorem}
\subsection{Linear Convergence} 
 \label{sec-linear-conv}
We use the following assumption which was also adopted in \cite{TsengIncrGradient98,Solodov98IncrGrad,AlgEkfs2003,schmidt2013fast} for analyzing stochastic and incremental gradient methods. 
\begin{assumption}\label{assum-gradfi} \textbf{(Gradient growth condition)} There exists a positive constant $M$ such that $$\|\gradfi(x)\| \leq M \|\nabla f(x)\|$$ for all $i=1,2,\dots,m$. 
\end{assumption}
Assumption \ref{assum-gradfi} states that the norm of $ \nabla f_1,\nabla f_2,\dots, \nabla f_m $ is bounded by a linear function of the norm of $\nabla f$. Thus, it limits the oscillations that might arise due to an imbalance between the norm of $\nabla f_i$ (for some $i$) and the norm of $\nabla f$ which led us previously to adopt a variable stepsize that gets smaller when such oscillations arise (see the paragraph after Assumption \ref{assum-stepsize}). Indeed, we show in Theorem \ref{thm-lower-bound-stepsize-growth} that this assumption, by limiting such oscillations, can avoid the variable stepsize rule of Assumption \ref{assum-stepsize} getting too small (keeping the normalized stepsize bounded away from zero).

Note that Assumption \ref{assum-gradfi} requires $\nabla f_1(x) = \nabla f_2(x) = \cdots = \nabla f_m(x) = 0$ at a stationary point $x$ of $f$. This requirement, although restrictive, is not completely unrealistic for certain applications such as neural network training problems or non-linear least square problems when the residual error is zero~\cite{TsengIncrGradient98,schmidt2013fast}. Under this assumption,
 
\begin{itemize}
  \item Tseng \cite{TsengIncrGradient98} shows that his incremental gradient method is either linearly convergent or the stepsize is bounded away from zero in which case convergence rate is not known. It is not clear whether his method can achieve linear convergence under stronger conditions. This method is not applicable to our setting as it requires constructing and evaluating the full gradient, $\nabla f$.
  \item Solodov \cite{Solodov98IncrGrad} shows global convergence of the IG method with a constant stepsize, although no convergence rate results are given.
  \item Moriyama \textit{et al.} \cite{AlgEkfs2003} show that EKF-S method is globally convergent with a stepsize $\alpha^k$ that grows linearly with $k$ but do not provide an explicit lower bound on the growth rate or any convergence rate results.   
  \item Schmidt \cite{schmidt2013fast} proves that the SGD method is linearly convergent in expectation in a stochastic setting but the analysis and results are not deterministic and do not apply directly to our method.  
\end{itemize} 

In this paper, we show the linear convergence of our method under Assumption \ref{assum-gradfi} when the stepsize control parameter $\eta$ in Assumption \ref{assum-stepsize} is appropriately chosen. As a by-product, our analysis also implies the linear convergence of the EKF-S algorithm (see Corollary \ref{coro-linear-conv-EKF}) which was left open in \cite[\S 7]{AlgEkfs2003}.

By adapting a result from Moriyama \textit{et al.}~\cite[Theorem 4.2]{AlgEkfs2003}, it is not hard to show that $\gammalb$ is positive under Assumption \ref{assum-gradfi}. However, Moriyama \textit{et al.} do not provide an explicit lower bound on $\gammalb$ (in terms of the problem constants $m$, $c$, $C$ and $M$). For estimating an explicit lower bound, we will need the following lemma.
 
\begin{lemma}\label{lemm-Bi-growth} Suppose that Assumptions \ref{assum-compact}, \ref{assum-strong-cvx}, \ref{assum-stepsize} and \ref{assum-gradfi} hold. Let the iterates $\{x_1^k\}_{k=1}^\infty$ be generated by the IN method \eqref{eq-inner-update}--\eqref{eq-hessian-update-convention}. Then, for each $i \in \{1,2,\dots,m\}$, we have
\beq\label{ineq-inner-iter-distance-bound-by-grad} \| x_i^k - x_1^k \| \leq \gamma^k \bik(\phi)  \|\nabla f(x_1^k)\|, \quad \mbox{for all} \quad k\geq 2,
\eeq 
where $\bik(\phi)$ is given by the recursion \beq B_{i+1}^k(\phi) =  \bigg(1+ \frac{2Q}{m} \max(1/k,\phi)\bigg)\bik(\phi) + \frac{2M}{cm} \quad \mbox{and} \quad B_1^k = 0,\label{eq-Bik-phi} \eeq 
and the limits $ B_i(\phi) : = \lim_{k \to \infty} B_i^k(\phi) $
satisfy 
 \beq\label{eq-recursive-B-phi} B_{i+1}(\phi) =  \bigg(1+ \frac{2Q}{m}\phi\bigg)B_i(\phi) + \frac{2M}{cm} \quad \mbox{and} \quad B_1(\phi) = 0,
 \eeq 
where $\phi$ is given by \eqref{def-phi}.
\end{lemma}
\begin{proof} Fix $k$. We will proceed by induction on $i$. For $i=1$, the left-hand side of \eqref{ineq-inner-iter-distance-bound-by-grad} is zero, so the result holds. For simplicity of the notation, we will write $B_i^k$ for $B_i^k(\phi)$. Suppose that \eqref{ineq-inner-iter-distance-bound-by-grad} holds for all $1 \leq i \leq j \leq m$. Then,
\beqa
	\|x_{j+1}^k - x_1^k \| 
	&\leq& B_j^k \gamma^k \|\nabla f(x_1^k)\|  + \|x_{j+1}^k - x_j^k \| \nonumber \\
	&\leq & B_j^k \gamma^k \|\nabla f(x_1^k)\|  + \frac{2\gamma^k}{cm} \big\| \nabla f_j (x_j^k)\big\|  \nonumber \\
	&\leq & B_j^k \gamma^k \|\nabla f(x_1^k)\|  + \frac{2\gamma^k }{cm} \bigg( M \big\| \nabla f (x_1^k)\big\| + \hessub \big\| x_j^k - x_1^k \big\| \bigg) \nonumber \\
	&\leq & B_j^k \gamma^k \|\nabla f(x_1^k)\|  + \frac{2\gamma^k }{cm} \bigg( M \big\| \nabla f (x_1^k)\big\| + \hessub B_j^k \gamma^k \|\nabla f(x_1^k)\| \bigg) \nonumber \\
	&= & \bigg( B_j^k \big(1 + \frac{2Q}{m} \gamma^k\big) + \frac{2M}{cm}\bigg) \gamma^k \|\nabla f(x_1^k)\|, \label{ineq-inner-dist-bd-by-grad-helper}
\eeqa
where we used the induction hypothesis in the first and the fourth inequality, the inner update equation \eqref{eq-inner-update} for relating the distance between inner iterates to gradients and Lemma \ref{eq-iter-hess-bound} for bounding the norm of $(H_j^k)^{-1}$ in the second inequality, and the third inequality follows from Assumption \ref{assum-gradfi} on the gradient growth, the triangle inequality over the gradients and \eqref{eq:gradient-Lipschitzness} on the Lipschitzness of the gradients. Using Assumption \ref{assum-stepsize} on the variable stepsize and the bound on the normalized stepsize from Lemma \ref{lemm-alphastar-bound}, we have for any $k\geq 1$, 
 $$\gamma^k \leq  \max(1/k,\gamma_*^k)\leq \max(1/k, \phi),$$
which, once combined with \eqref{ineq-inner-dist-bd-by-grad-helper}, implies that the inequality \eqref{ineq-inner-iter-distance-bound-by-grad} is true for $i=j+1$. This completes the induction-based proof of the equality \eqref{eq-Bik-phi}. Then, \eqref{eq-recursive-B-phi} follows directly from \eqref{eq-Bik-phi} by taking the limit as $k\to \infty$ and using the fact that $\phi>0$. \qed
\end{proof}

We use the preceding result to provide a lower bound on the asymptotic behavior of the normalized stepsize. 
\begin{theorem}\label{thm-lower-bound-stepsize-growth} \textbf{(Asymptotic stepsize behavior)}
Suppose that Assumptions \ref{assum-compact}, \ref{assum-strong-cvx}, \ref{assum-stepsize} and \ref{assum-gradfi} hold. Then, there exists a constant $\kappa$ such that 
\beq\label{ineq-lower-bound-on-stepsize-growth} 0 < \kappa  \leq \gammalb.
\eeq
Furthermore, if $\phi < \frac{1}{B(\phi)C}$ where $\phi$ is defined by \eqref{def-phi}, 
  \beq\label{def-B-phi-as-a-sum} B(\phi):=\sum_{j=2}^m B_j(\phi) \eeq 
and $B_j(\phi)$ is given by \eqref{eq-Bik-phi}, then a choice of  
\beq\label{eq-liminf-stepsize-bound}
\kappa = \phi \frac{1}{Q^2} \frac{1}{ \frac{2B(\phi)C}{1-B(\phi)C\phi} + 1} 
\eeq
satisfies \eqref{ineq-lower-bound-on-stepsize-growth}.
\end{theorem}
\begin{proof}
The existence of such $\kappa$ follows by a reasoning along the lines of \cite[Theorem 4.2]{AlgEkfs2003}. To prove the second part, suppose that $\phi < \frac{1}{B(\phi)C}$ and Assumptions \ref{assum-compact}, \ref{assum-strong-cvx}, \ref{assum-stepsize} and \ref{assum-gradfi} hold.  
Using the definition of $h_m^k$ from \eqref{eq-iter-ver-compact2},
\beqa
   \| \nabla f(x_1^k) - h_m^k \| &\leq& \sum_{j=1}^m \bigg( \| \nabla f_j(x_j^k) - \nabla f_j(x_1^k) \| + \frac{1}{\alpha^k} \| \nabla^2 f_j(x_j^k)(x_1^k - x_j^k) \| \bigg) \nonumber \\
   &\leq& C( 1 + 1/\alpha^k) \sum_{j=2}^m  \|x_j^k - x_1^k\|, \label{ineq-grad-approx-error} 
\eeqa
where we used \eqref{eq-hessian-bounds} and \eqref{eq:gradient-Lipschitzness} in the second inequality. Let $k\geq 2$. Using Lemma \ref{lemm-Bi-growth}, we have 
\beq\label{ineq-distance-bd-via-grad-growth-bound}
	\sum_{j=2}^m  \|x_j^k - x_1^k\| \leq \gamma^k B^k(\phi)  \| \nabla f(x_1^k)\| 
\eeq
where \beq\label{def-Bk-phi}
B^k(\phi):=\sum_{j=2}^m B_i^k(\phi); \quad \lim_{k \to \infty} B^k(\phi)=B(\phi).
\eeq 
Combining \eqref{ineq-grad-approx-error} and \eqref{ineq-distance-bd-via-grad-growth-bound}, 
\beqa
	\| h_m^k \| &\geq& \| \nabla f(x_1^k) \| - \| \nabla f(x_1^k) - h_m^k \|  \nonumber \\
	 &\geq& C \bigg( \frac{1}{ B^k(\phi) C\gamma^k} - (1+1/\alpha^k)\bigg)  \sum_{j=2}^m  \|x_j^k - x_1^k\| \nonumber \\
	 &=& C \bigg( \frac{1}{ B^k(\phi) C \gamma^k} - \big(1+\frac{1}{k\gamma^k}\big)\bigg)  \sum_{j=2}^m  \|x_j^k - x_1^k\|, \label{ineq-gradstep-bound} 
\eeqa
where $\gamma^k$ is the normalized stepsize given by \eqref{eq-normalized-step}. By assumption $ B(\phi) C \phi < 1$. Using Lemma \ref{lemm-alphastar-bound}, this implies that $ B(\phi)C \gammaub <1$. Since $\gammalb > 0$ by the first part, the right hand side of \eqref{ineq-gradstep-bound} stays positive for $k$ large enough. Then, by re-arranging \eqref{ineq-gradstep-bound}, there exists $\bar{k}$ such that for $k\geq \bar{k}$,
\beqas
I_k := \frac{\sum_{j=2}^m  \|x_j^k - x_1^k\| }{\| h_m^k \|} 
	&\leq& \frac{B^k(\phi)\gamma^k}{1 - B^k(\phi)C\gamma^k - B^k(\phi) C/k}.
 \eeqas
Thus, 
\beq\label{ineq-Ik-bound}
	\limsup_{k\to\infty} \frac{I_k }{\gamma^k} \leq \frac{B(\phi)}{1-B(\phi)C\gammaub} \leq \frac{B(\phi)}{1-B(\phi)C\phi}, 
\eeq
where we used Lemma \ref{lemm-alphastar-bound} to bound $\gammaub$ and \eqref{def-Bk-phi} for taking the limit superior of the sequence $\{B^k(\phi)\}$. We also have 
\beqa
	\alpha_*^k 
	&\geq&  \frac{1-\eta}{\hessub} \frac{\hesslb mk (\alpha^k)^2 \| D_m^k h_m^k \|^2}{\alpha^k \| D_m^k h_m^k\| \sum_{j=2}^m \|x_j^{k} - x_1^k\| + (m/2) (\alpha^k)^2  \| D_m^k h_m^k\|^2} \nonumber \\
	&\geq& \frac{(1-\eta)}{\hessub} \frac{\hesslb mk}{I_k /(\alpha^k/Cmk) + m/2} 
	= \bigg(2\frac{(1-\eta)}{Q} \frac{1}{2C (I_k /\gamma^k) + 1}\bigg) k, \label{ineq-gamma-star-k-lb}
\eeqa
where we used Lemma \ref{eq-iter-hess-bound} to bound $H_m^k$ and $D_m^k$ in the first and second inequalities respectively. Combining \eqref{ineq-Ik-bound} and \eqref{ineq-gamma-star-k-lb} and letting $k \to \infty$,
\beqas 
  \gammalb \geq 2\frac{(1-\eta)}{Q} \frac{1}{2C \limsup_{k\to\infty}(I_k /\gamma^k) + 1} 
    &\geq &  \phi \frac{1}{Q^2} \frac{1}{\frac{2B(\phi)C}{1-B(\phi)C\phi} + 1} > 0,
\eeqas
which is the desired result. \qed
\end{proof}
\begin{remark}\label{rema-min-B-bound} Since $\eta \in (0,1)$, we have $\phi \in (0,2Q)$ by the definition \eqref{def-phi} of $\phi$. Taking limits in \eqref{eq-recursive-B-phi}, we obtain
  $$\lim_{\phi \to 0} B_j(\phi) = \frac{2M}{cm}(j-1), \quad \mbox{for} \quad j=1,2,\dots,m.$$
Then, as $B_j(\phi)$ is a monotonically non-decreasing function of $\phi$ for every $j$ (see \eqref{eq-recursive-B-phi}), $B(\phi)$ defined by \eqref{def-B-phi-as-a-sum} is also non-decreasing in $\phi$ satisfying,
  $$B_{min} := \inf_{\phi \in(0,2Q) } B(\phi) = \lim_{\phi \to 0} B(\phi) = \sum_{j=2}^m\frac{2M}{cm}(j-1) = \frac{M(m-1)}{c}  > 0.$$
This shows that $B(\phi)C$ can never vanish so that the condition $\phi < \frac{1}{B(\phi)C}$ in Theorem \ref{thm-lower-bound-stepsize-growth} is well-defined. To see that this condition is always satisfied when $\phi$ is positive and small enough, note that the monotonicity of $B(\phi)$ leads to
   $$ B_{max}:=\sup_{\phi \in(0,2Q) } B(\phi) = \lim_{\phi \to 2Q} B(\phi) < \infty$$
as well. Hence, $\phi \in (0,\frac{1}{B_{max} C})$ always satisfies this condition. 
\end{remark}
\medskip

We now analyze Algorithm IN as an inexact perturbed Newton method. 
Using the representation \eqref{eq-iter-ver-compact} and the formula \eqref{eq-hessian-iterates-decompositon} for the Hessian matrices at the iterates, we can express IN iterations as 
\beq\label{eq-iter-newton-via-error} x_1^{k+1} = x_1^k - \gamma^k (\bar{H}_k)^{-1}\big(\nabla f(x_1^k) + e^k \big) 
\eeq
where \beq\label{def-averaged-hessian} \bar{H}_k := \frac{H_m^k}{k} =  \frac{\sum_{i=1}^k \big( \sum_{j=1}^m \nabla^2 f_j(x_j^i) \big) }{k} = \frac{\sum_{i=1}^k  \nabla^2 f(x_1^i)  }{k} + \widehat{e}^k
\eeq
is the average of the Hessian of $f$ at the previous iterates up to an error term 
\beq\label{eq-hessian-error-in-IN} \widehat{e}^k = \frac{\sum_{i=1}^k  \sum_{j=1}^m \bigg( \nabla^2 f_j(x_j^i) - \nabla^2 f_j(x_1^i)\bigg) }{k}, 
\eeq 
and the gradient error is 
\beq\label{eq-grad-error-alg2}
e^k = \sum_{j=1}^m \bigg(\nabla f_j (x_j^k)-\nabla f_j (x_1^k) + \frac{1}{\alpha^k} \nabla^2 f_j(x_j^k)(x_1^k - x_j^k) \bigg). 
\eeq
Applying \eqref{eq-hessian-bounds} on the Hessian bounds to the first equality in \eqref{def-averaged-hessian}, the Hessian term $\bar{H}_k $ satisfies  
   \beq\label{eq-averaged-hessian-bounds} \hesslb mI \preceq \bar{H}_k \preceq \hessub mI, 
   \eeq 
where $I$ is the $n\times n$ identity matrix and the Hessian error $\widehat{e}^k$ admits the simple bound
  \beq\label{ineq-hessian-error-trivial-bd} \|\widehat{e}^k\| \leq (\hessub - \hesslb)m=cm(Q-1).
  \eeq
We can also bound the gradient error in terms of the norm of the gradient for $k\geq 2$ as
\begin{eqnarray*}
\|e^k\| &\leq & \sum_{j=1}^m \bigg(\| \nabla f_j (x_j^k)-\nabla f_j (x_1^k)\| + \frac{1}{\alpha^k}  \|\nabla^2 f_j(x_j^k)\| \|x_1^k - x_j^k \| \bigg) \\
&\leq& (C + \hessub/\alpha^k) \gamma^k \sum_{j=1}^m B_j^k(\phi)  \| \nabla f(x_1^k) \|  =  {C(\gamma^k + 1/k)} B^k(\phi) \|\nabla f(x_1^k) \|,
\end{eqnarray*}
where we used \eqref{eq-hessian-bounds} on the boundedness of the Hessian matrices, \eqref{eq:gradient-Lipschitzness} on the Lipschitzness of the gradients and Lemma \ref{lemm-Bi-growth} on the distance between the iterates in the second inequality, the (last) equality holds by the definitions \eqref{eq-normalized-step} and \eqref{def-Bk-phi}. This leads to 
\beq\label{eq-limsup-grad-error}
	\limsup_{k\to\infty} \bigg(\| e^k \| / \|\nabla f(x_1^k) \|\bigg) \leq CB(\phi)\gammaub
	\leq CB(\phi)\phi,
\eeq
by Lemma \ref{lemm-alphastar-bound}. 

We prove our rate results using the next theorem regarding sufficient conditions for linear convergence of the inexact perturbed Newton methods of the form
\beqa\label{eq:iter-perturbed-inexact-newton-1}
	(F'(y^k) + \Delta_k)s^k &=& -F(y^k) +\delta_k \\
	y^{k+1} &=& y^k + s^k, \quad y^1 \in \R^n, \label{eq:iter-perturbed-inexact-newton-2}
\eeqa
where the map $F :  \R^n \to \R^n$ and $F'$ denotes the Jacobian matrix of $F$, $\delta^k$ is the perturbation to $F$ and $\Delta^k$ is the perturbation to the Jacobian matrix $F'$. The local convergence of such iterates to a solution $y^*$ satisfying $F(y^*)=0$ is well-studied. Under the following conditions, 
\begin{itemize}
    \item  there exists $y^*$ such that $F(y^*)=0$, \hfill (C1)	
	\item The Jacobian matrix $F'(y^*)$ is invertible, \hfill (C2)
	\item $F$ is differentiable on a neighborhood of $y^*$ and $F'$ is continuous at $y^*$, \hfill (C3)
	\item $\Delta_k$ are such that $F'(y^k) + \Delta_k$ are non-singular for all $k=1,2,\dots$, \hfill (C4)
\end{itemize}
the following local linear convergence result is known in the literature.
\begin{theorem}\label{thm-lin-conv-newton-method} \textbf{(\cite[Theorem 2.2]{Catinas2001Jota}, based on \cite{DemboEisenNewton82})} Assume conditions \emph{(C1)--(C4)} are satisfied. Given $0\leq \xi_k \leq \bar{\xi}<t<1$, $k=1,2,\dots$, there exists $\epsilon>0$ such that if $\| y^1 - y^*\|\leq \epsilon$  and if the iterates $\{y^k\}$ generated by \eqref{eq:iter-perturbed-inexact-newton-1}--\eqref{eq:iter-perturbed-inexact-newton-2} satisfy 
$$ \left\| \Delta_k \bigg( F'(y^k)+\Delta_k \bigg)^{-1} F(y^k) + 
  \bigg( I - \Delta_k ( F'(y^k) + \Delta_k )^{-1} 
  \bigg) \delta_k  \right\| \leq \xi_k \| F(y^k) \|,
$$ 
then the convergence is linear in the sense that 
 $$ \| y^{k+1} - y^* \|_* \leq t \| y^k - y^*\|_*, \quad k=1,2,\dots. $$
 where $\| z \|_* := \| F'(y^*) z \|$.
\end{theorem}

When $\Delta_k = 0$, Theorem \ref{thm-lin-conv-newton-method} says roughly that as long as the perturbation $\delta_k$ is a \textit{relative error} in the right-hand side of \eqref{eq:iter-perturbed-inexact-newton-1}, i.e. $\|\delta_k\|/\| F(y^k) \| \leq \xi_k< \bar{\xi}<1$, we have local linear convergence. This result was first proven by Dembo \textit{et al.} \cite{DemboEisenNewton82} and was later extended to the $\Delta_k \neq 0$ case in \cite{Catinas2001Jota} by noticing that 
\eqref{eq:iter-perturbed-inexact-newton-1} can be re-arranged and put into the form 
 $$F'(y^k)s^k = -F(y^k) + \bar{\delta}_k$$
with 
 $$ \bar{\delta}_k = \Delta_k \bigg( F'(y^k)+\Delta_k \bigg)^{-1} F(y^k) + 
  \bigg( I - \Delta_k ( F'(y^k) + \Delta_k )^{-1} 
  \bigg) \delta_k $$
(see \cite[Theorem 2.2]{Catinas2001Jota}). Then, the result of Dembo \textit{et al.} for the $\Delta_k=0$ case is directly applicable, leading to a proof of Theorem \ref{thm-lin-conv-newton-method}.  

\bigskip
In order to be able to apply Theorem \ref{thm-lin-conv-newton-method}, we rewrite the IN iteration given by \eqref{eq-iter-newton-via-error} in the form of \eqref{eq:iter-perturbed-inexact-newton-1}--\eqref{eq:iter-perturbed-inexact-newton-2} by setting
\beqas
	&&y^k = x_1^k, \quad F(\cdot)=\nabla f(\cdot), \quad \delta_k = -e^k, \quad y^* = x^*, \\
	&&F'(\cdot) = \nabla^2 f(\cdot),\quad \Delta_k = (\bar{H}_k/\gamma^k) - \nabla^2 f(x_1^k).     
\eeqas
In this formulation, the perturbation $\delta^k$ and the gradient error $e^k$ are equal up to a sign factor. The additive Hessian error $\widehat {e}^k$ is similar to $\Delta^k$ but is not exactly the same because $\Delta^k$ has also a division by the normalized stepsize in it. 
  
Note that the previous lower bound on $\kappa$ obtained in Theorem \ref{thm-lower-bound-stepsize-growth} for the growth rate of $\alpha_*^k$ is achieved in the limit as $k$ goes to infinity. But, $\alpha_*^k$ can achieve any growth rate strictly less than $\kappa$, say $\nu k$ for some $\nu \in (0,1)$, in finitely many steps. To capture such growth in stepsize for some finite $k$, we define the following stepsize rule that satisfies Assumption \ref{assum-stepsize}. 
\noindent \begin{assumption}\label{assum-stepsize-simple-with-lin-conv} \textbf{(Variable stepsize with linear growth)} The stepsize $\alpha^k$ depends on the parameters $(\widehat\eta,\widehat\nu,\widehat\kappa)$ and satisfies 
\beqa\label{stepsize-linear-growth} \alpha^k = \alpha^k(\widehat\eta,\widehat\nu,\widehat\kappa) = 
\begin{cases} 
      (\widehat\nu \widehat\kappa) k,   \hfill & \text{ if } 1\leq (\widehat\nu \widehat\kappa) k \leq \max\big(1,\alpha_*^k \big) \\
      1, & \text{ otherwise}, 
  \end{cases}
\eeqa  
for some $\widehat\kappa > 0$ and $\widehat\eta,\widehat\nu \in (0,1)$ where $\alpha_*^k$ is defined by \eqref{def-alpha-star-k} with stepsize control parameter $\eta$ equal to $\widehat{\eta}$.
\end{assumption} 

We argue now how this choice of stepsize with linear growth can lead to linear convergence if parameters $(\widehat\eta,\widehat\nu,\widehat\kappa)$ are chosen appropriately. Suppose that Assumptions \ref{assum-compact} and \ref{assum-strong-cvx} on the boundedness of iterates and Hessian matrices, Assumption \ref{assum-stepsize-simple-with-lin-conv} on the variable stepsize with parameters $(\eta,  \nu,\kappa)$ where $\eta,\nu \in (0,1)$ and $\kappa$ is given by \eqref{eq-liminf-stepsize-bound} and Assumption \ref{assum-gradfi} on the gradient growth hold. Suppose also that\footnote{Remark \ref{rema-min-B-bound} shows that this condition is always satisfied when $\phi$ is small enough. It is adopted both to use the $\kappa$ bound from Theorem \ref{thm-lower-bound-stepsize-growth} (by having $\phi < 1 / \big( B(\phi)C \big)$ and also to keep the normalized stepsize small enough (by having $\phi < 1/Q$ which implies $\gamma^k < 1/Q$ by Lemma \ref{lemm-alphastar-bound}) to control the norm of the perturbations in the estimates \eqref{ineq-limsup} and \eqref{ineq-limsup-inv-newt-error}.} 
\beq\label{condition-on-phi}\phi < \min\bigg(\frac{1}{Q},\frac{1}{B(\phi)C}\bigg)
\eeq where $B(\phi)$ is given by \eqref{def-B-phi-as-a-sum}. Using Lemma \ref{lemm-alphastar-bound} and Theorem \ref{thm-lower-bound-stepsize-growth}, $\alpha_*^k$ grows linearly with an asymptotic rate bounded from below by $\kappa$ and bounded from above by $\phi$ so that the stepsize defined by \eqref{stepsize-linear-growth} with parameters $(\eta,\nu,\kappa)$ satisfies 
 \beq\label{eq-limit-normalized-stepsize-nu-kappa} \lim_{k \to \infty} \gamma^k = \lim_{k \to \infty} \frac{\alpha^k}{k} = \nu \kappa < \limsup_{k \to \infty} \frac{\alpha_*^k}{k} \leq \phi.
    \eeq
%
Note that by \eqref{eq-hessian-bounds} and \eqref{eq-averaged-hessian-bounds} on the boundedness of the Hessian matrices $\nabla^2 f(y^k)$ and the averaged Hessian $\bar{H}_k$, we have 
    \beq\label{bounds-hessian-ratio} \frac{1}{Q}I \preceq \nabla^2 f(y^k) \bar{H}_k^{-1} \preceq Q I. 
    \eeq 
As $\phi < 1/Q$ by the condition \eqref{condition-on-phi}, the inequalities \eqref{eq-limit-normalized-stepsize-nu-kappa} and \eqref{bounds-hessian-ratio} imply that there exists a positive integer $\bar{k}$ such that for $k\geq \bar{k}$, we have $\gamma^k < 1/Q$ and 
\beq\label{ineq-step-with-hessian-bounds} 0 \prec \bigg(1 - \gamma^k Q\bigg) I \preceq \bigg( I - \gamma^k \nabla^2 f(y^k) \bar{H}_k^{-1}\bigg) \preceq \bigg(1 - \frac{\gamma^k}{Q}\bigg)I.
\eeq
Combining \eqref{bounds-hessian-ratio} and \eqref{ineq-step-with-hessian-bounds} leads to 
\beqa\label{ineq-limsup-error-hessian}
\limsup_{k\to\infty} \bigg\| \Delta_k \bigg( F'(y^k)+\Delta_k \bigg)^{-1}\bigg\| &=& \limsup_{k\to\infty} \bigg\| I - \gamma^k \nabla^2 f(y^k) \bar{H}_k^{-1} \bigg\| \nonumber \\
&\leq & 1 - \frac{\liminf_{k \to \infty} \gamma^k}{Q} = 1 - \frac{\nu\kappa}{Q} \label{ineq-limsup}
\eeqa
and similarly to 
\beqa
	\limsup_{k \to \infty} \left\| \bigg( I - \Delta_k \bigg( F'(y^k)+\Delta_k \bigg)^{-1}  \bigg) \right\| &=& \limsup_{k \to \infty}  \bigg\|\gamma^k \nabla^2 f(y^k) \bar{H}_k^{-1} \bigg\| \nonumber \\
&\leq &  \phi Q \label{ineq-limsup-inv-newt-error}
\eeqa
where we used \eqref{eq-limit-normalized-stepsize-nu-kappa} for bounding the limit superior of $\gamma^k$ and \eqref{bounds-hessian-ratio} to bound the norm of the matrix products.
Hence,
\beqa
	\limsup_{k \to \infty} \frac{\left\| \bigg( I - \Delta_k \bigg( F'(y^k)+\Delta_k \bigg)^{-1}  \bigg)\delta_k\right\|}{\|F(y^k)\|} &\leq& \phi Q \limsup_{k \to \infty} (\|\delta^k\|/\|F(y^k)\|) \nonumber \\
	&\leq& \phi^2 QB(\phi)C \label{ineq-limsup-grad-error-inexact-newton}  
\eeqa
where we used \eqref{ineq-limsup-inv-newt-error} in the first inequality and \eqref{eq-limsup-grad-error} in the second inequality. Combining \eqref{ineq-limsup-error-hessian} and \eqref{ineq-limsup-grad-error-inexact-newton}, 
\beqa \eta^\infty :&=& \limsup_{k\to\infty} \frac{\left\| \Delta_k \bigg( F'(y^k)+\Delta_k \bigg)^{-1} F(y^k) + 
  \bigg( I - \Delta_k ( F'(x_k) + \Delta_k )^{-1} 
  \bigg) \delta_k  \right\|}{\|F(y^k)\|} \nonumber \\
  &\leq & 1 - \frac{\nu\kappa}{Q}   
  + \phi^2 QB(\phi)C.  \nonumber \\
  &=& 1 - \phi \frac{\nu}{Q^3} \frac{1}{\frac{2B(\phi)C}{1-B(\phi)C\phi} + 1} + \phi^2 QB(\phi)C := r_{\nu}(\phi),  \label{ineq-limiting-rate}
\eeqa 
where we used the definition of $\kappa$ from \eqref{eq-liminf-stepsize-bound} in the last equality. By Assumption \ref{assum-strong-cvx} on the strong convexity and the regularity of $f$, the conditions (C1)--(C4) are satisfied for $F(\cdot)=\nabla f(\cdot)$. Hence, we can apply Theorem \ref{thm-lin-conv-newton-method}, which says that it suffices to have \beq\label{cond-linear-conv-r-less-than-1} r_{\nu}(\phi) < 1\eeq for local linear convergence. It is straightforward to see from \eqref{ineq-limiting-rate} that this condition is satisfied for $\phi$ positive and around zero as $r_{\nu}(0)=1$ and the derivative $r_{\nu}'(0)<0$. Remembering the assumption \eqref{condition-on-phi}, we conclude that there exists a positive constant $\overline{\phi}_{\nu}$ such that we have linear convergence when
\beq\label{ineq-neces-cond-for-lin-conv}
0 < \phi < \overline{\phi}_{\nu} \leq \min\bigg(\frac{1}{Q}, \frac{1}{B(\phi)C}\bigg).
\eeq
We discuss how $\overline{\phi}_{\nu}$ can be determined later in Remark \ref{rema-upper-bound-phi-bar}. This condition for linear convergence \eqref{ineq-neces-cond-for-lin-conv} is satisfied if 
\beq\label{cond-on-eta-for-lin-conv}
 0 < 1 - \eta < \frac{\overline{\phi}_{\nu}}{2Q},
\eeq
by the definition of $\phi$ in \eqref{def-phi}. Thus, by choosing the stepsize control parameter $\eta \in (0,1)$ close enough to $1$, we can satisfy \eqref{cond-on-eta-for-lin-conv} and hence guarantee local linear convergence of the IN algorithm. These findings provide a proof of the following linear convergence result.

\begin{theorem}\label{thm-local-lin-conv-with-stepsize} \textbf{(Linear convergence with variable stepsize)} Suppose that Assumptions \ref{assum-compact}, \ref{assum-strong-cvx} and \ref{assum-gradfi} hold. Let $\nu \in (0,1)$ be given and the stepsize control parameter $\eta \in (0,1)$ satisfy the inequality \eqref{cond-on-eta-for-lin-conv}. Then, the IN method with the stepsize rule \eqref{stepsize-linear-growth} defined by Assumption \ref{assum-stepsize-simple-with-lin-conv} with parameters $(\eta,\nu,\kappa)$ where $\kappa$ is given by \eqref{eq-liminf-stepsize-bound}
is locally linearly convergent with rate $t$ where $t$ is any number satisfying
$$0 \leq r_{\nu}(\phi) < t < 1,$$
$\phi$ is defined by \eqref{def-phi} and $r_{\nu}(\phi)$ is given by \eqref{ineq-limiting-rate}.
\end{theorem}

Our arguments and proof techniques apply also to the EKF method with variable stepsize rule (EKF-S algorithm) which also uses the stepsize defined by Assumption \ref{assum-stepsize} and makes similar assumptions such as the boundedness of iterates and Lipschitzness of the gradients for achieving global convergence \cite{AlgEkfs2003}. Our reasoning with minor modifications leads to the following linear convergence result whose proof is skipped due to space considerations. 
\begin{corollary}\label{coro-linear-conv-EKF} \textbf{(Linear convergence of EKF-S)} Consider the problem \eqref{pbm-multi-agent} with $f_i = \frac{1}{2} g_i^2$ where each $g_i:\R^n \to \R$ is continuously differentiable for $i=1,2,\dots,m$ (non-linear least squares). Consider the EKF-S algorithm of Moriyama \textit{et al.} with variable stepsize $\alpha^k$ \cite{AlgEkfs2003}. 
Let $\nu \in (0,1)$ be given and $\eta \in (0,1)$. Suppose that Assumptions 3.1--3.2, 3.4--3.5 and 4.1 from Moriyama \textit{et al.} \cite{AlgEkfs2003} hold. It follows that there exists constants $\tilde \kappa>0$ and $\tilde \phi_\nu > 0$ such that if 
  $$ 0 <1- \eta < \tilde \phi_\nu,$$
then the EKF-S algorithm with the variable stepsize rule \eqref{stepsize-linear-growth} with parameters $(\eta,\nu,\tilde \kappa)$ given in Assumption \ref{assum-stepsize-simple-with-lin-conv} is locally linearly convergent. 
\end{corollary}

We continue by several remarks about satisfying the variable stepsize rules we introduced by a simple adaptive stepsize algorithm, improving the convergence rate results with additional regularity assumptions, extensions to incremental quasi-Newton methods, some advantages of an incremental Newton approach over incremental quasi-Newton approaches in our framework and determining the constant $\overline{\phi}_{\nu}$.  

\begin{remark}\label{rema-adaptive-stepsize} By definition, the exact value of $\alpha_*^k$ can only be computed at the end of the $k$-th iteration as it requires the knowledge of $H_m^k$. However, it is possible to guarantee that Assumption \ref{assum-stepsize} on the variable stepsize holds by a simple bisection-type adaptive algorithm as follows: 
\begin{enumerate} 
\item [] \textit{\underline{Adaptive stepsize with bisection:}}
\item  At the start of the $k$-th cycle, i.e. right after $x_1^k$  is available, set $\alpha^k$ to an initial value, say $\alpha^k(j)$ with $j=1$. 
\item  Compute $\alpha_*^k$ (depending on $\alpha^k(j)$) by running one cycle of the algorithm.
\item  If Assumption \ref{assum-stepsize} is satisfied, accept the step, set $\alpha^k=\alpha^k(j)$ and exit. Else, bisect by setting $\alpha^k(j+1)=\max\big(1,\tau \alpha^k(j)\big)$ for some $\tau \in (0,1)$, increment $j$ by 1 and go to step 2.
\end{enumerate} 
There is no risk of an infinite loop during the bisections, as the step $\alpha^k=1$ is always accepted. Theorem \ref{thm-lower-bound-stepsize-growth} shows that when Assumption \ref{assum-gradfi} on the gradient growth holds, by setting $\alpha^k(1) = (\nu \kappa) k$ in the above iteration as in the stepsize rule  \eqref{stepsize-linear-growth} with $\nu \in (0,1)$ and $\kappa$ as in \eqref{eq-liminf-stepsize-bound}, $\alpha^k(1)$ will be immediately accepted requiring no bisection steps, except for finitely many $k$. 
\end{remark}

\begin{remark}\label{rema-quasi-newton-extension} Main estimates used for proving Theorem \ref{thm-local-lin-conv-with-stepsize} are the inequalities \eqref{ineq-limsup} and \eqref{ineq-limsup-inv-newt-error} which require only \eqref{eq-averaged-hessian-bounds} on the boundedness of the averaged Hessian. If one replaces actual Hessians $\nabla^2 f_i(x_i^k)$ with approximate Hessians $\nabla^2 \tilde{f}_i(x_i^k)$ as long as the eigenvalues of $\nabla^2 \tilde{f}_i(x_i^k)$ are bounded (with the same constant for all $i$ and $k$) by a lower bound $\tilde{c}>0$ and an upper bound $\tilde{C}>0$, all these inequalities as well as Theorem \ref{thm-local-lin-conv-with-stepsize} would still hold with $c$ and $C$ replaced by $\tilde{c }$ and $\tilde{C}$. Thus, the IN method admits straightforward generalizations to the incremental quasi-Newton methods while preserving its global convergence and linear convergence results. 
\end{remark}
\begin{remark}\label{rema-hessian-accurate-at-limit}
In the setting of Theorem \ref{thm-local-lin-conv-with-stepsize}, if we assume slightly more regularity than the continuity of the Hessian of $f$ implied by Assumption \ref{assum-strong-cvx}, the Hessian error upper bound \eqref{ineq-hessian-error-trivial-bd} and convergence rates can be improved as follows: Assume that the Hessian of $f$ is not only continuous but also H\"older continuous on the compact set $\mathcal{X}$ defined in Assumption \ref{assum-compact} with some H\"older exponent $\lambda$ and H\"older coefficient $L_\lambda$ satisfying $0<\lambda \leq 1$ and $L_\lambda<\infty$ (reduces to the Lipschitz condition if $\lambda=1$), then the Hessian error bound $\widehat{e}^k$ defined in \eqref{eq-hessian-error-in-IN} satisfies
\begin{eqnarray}  \| \widehat{e}^k \| &\leq& \frac{\sum_{i=1}^k  \sum_{j=1}^m L_\lambda \| x_j^i - x_1^i \|^{\lambda}}{k} \nonumber \\
 & \leq & L_\lambda \frac{\sum_{i=1}^k  \sum_{j=1}^m  \big( 1 + B_j^i(\phi)\big) (1 + \gamma^i) \| \nabla f(x_1^i) \|^{\lambda}}{k} \nonumber \\
 & \leq & L_\lambda   \frac{\sum_{i=1}^k  \sum_{j=1}^m  \big( 1 + B^i(\phi)\big) (1 + \gamma^i) \| \nabla f(x_1^i) \|^{\lambda}}{k} \nonumber \\
 & = & L_\lambda m  \frac{\sum_{i=1}^k \big( 1 + B^i(\phi)\big) (1 + \gamma^i) \| \nabla f(x_1^i)\|^{\lambda}}{k}, \label{ineq-hessian-error-bound-Holder}
\end{eqnarray}
where we used the definition of H\"older continuity in the first inequality, Lemma \ref{lemm-Bi-growth} that provide bounds on the distances between inner iterates together with the fact that $z^\lambda \leq 1 + z$ for $z\geq 0$ and $0<\lambda \leq 1$ in the second inequality and the upper bound $B_j^i(\phi) \leq \sum_{j=1}^m B_j^i(\phi) = B^i(\phi)$ for all $j$ (by the non-negativity of $B_j^i(\phi)$) in the third inequality. Note that the summand term in \eqref{ineq-hessian-error-bound-Holder} satisfies
	\beq\label{ineq-hessian-error-holder-cond} \lim_{i \to \infty} \big( 1 + B^i(\phi)\big) (1 + \gamma^i) \| \nabla f(x_1^i)\|^{\lambda} = 0 
	\eeq
because the gradient term $\nabla f(x_1^i)$ goes to zero by Theorem \ref{thm-global-conv} on the global convergence, the sequence $\{B^i(\phi)\}_{i\geq 1}$ is bounded admitting $B(\phi)<\infty$ as a limit by \eqref{def-Bk-phi} and the normalized stepsize $\gamma^i$ is bounded for any of the stepsize rules we discuss, including Assumptions \ref{assum-stepsize} and \ref{assum-stepsize-simple-with-lin-conv} (see Lemma \ref{lemm-alphastar-bound}). 
Combining \eqref{ineq-hessian-error-bound-Holder} and \eqref{ineq-hessian-error-holder-cond} shows that the Hessian error $\widehat{e}^k$ goes to zero, i.e. 
    \beq\label{ineq-hessian-error-decay-to-zero} \| \widehat{e}^k \|  \to 0, 
    \eeq
improving the trivial bound \eqref{ineq-hessian-error-trivial-bd}. 
Using \eqref{def-averaged-hessian} and the global convergence of $\{x_1^k\}$ to the optimal solution $x^*$ by Theorem \ref{thm-global-conv}, this implies that
\beq\label{fact-hessian-conv}
\nabla^2 f_i(x_1^k) \to \nabla^2 f_i(x^*), \quad \bar{H}_k \to \nabla^2 f(x^*), \quad \nabla^2 f(y^k) \bar{H}_k^{-1} \to I,
\eeq
with $y^k=x_1^k$. This allows us to replace the upper bounds $\big(1 -  \frac{\nu\kappa}{Q}\big)$ in \eqref{ineq-limsup} and $\phi Q$ in \eqref{ineq-limsup-inv-newt-error} with $(1-\nu\kappa)$ and $\phi$ respectively, eliminating some of the $Q$ terms in the condition $r_\nu(\phi) < 1$ for linear convergence (see \eqref{ineq-limiting-rate} and \eqref{cond-linear-conv-r-less-than-1})  and leading to the less restrictive condition 
   $$ 1 - \phi \frac{\nu}{Q^2} \frac{1}{\frac{2B(\phi)C}{1-B(\phi)C\phi} + 1} + \phi^2 B(\phi)C : = \hat{r}_\nu(\phi) < 1 $$
for linear convergence as $\hat{r}_\nu(\phi) \leq r_\nu(\phi)$ (with equality only in the special case $Q=1$). This relaxed condition would not extend to many classes of incremental quasi-Newton methods (such as DFP and its variants) that uses approximations $\nabla^2 \tilde{f}_i$ instead of the true Hessian $\nabla^2 f_i$ because such methods do not lead to an asymptotically correct Hessian approximation, i.e. it is possible that $\nabla^2 \tilde{f} \not\to \nabla^2 f$ (see \cite[Section 4]{DennisMore74}) so that \eqref{fact-hessian-conv} does not always hold. In this sense, using an incremental Newton approach instead of an incremental quasi-Newton approach in our algorithm allows us to get stronger convergence results.    
\end{remark}

\begin{remark}\label{rema-upper-bound-phi-bar}
It is possible to compute $\overline{\phi}_\nu$ as follows. By a straightforward computation, the condition on linear convergence $r_\nu(\phi) < 1$ (see \eqref{cond-linear-conv-r-less-than-1}) where $r_\nu(\phi)$ is given by \eqref{ineq-limiting-rate} is equivalent to 
\beqa
r_\nu(\phi) < 1 &\iff& 0 < \phi < \frac{1}{B(\phi)C}\frac{\nu}{Q^4} \frac{1}{ \frac{2B(\phi)C}{1-B(\phi)C\phi} + 1} \label{ineq-polynomial-cond-for-lin-conv}\\
&&\mbox{  and  } \phi < \min\bigg(\frac{1}{Q}, \frac{1}{B(\phi)C}\bigg)\label{ineq-polynomial-cond-2}  \\
 &\iff& 0 < B(\phi)^2 C^2 \phi^2 - \big(2B(\phi)C+1+\psi\big)B(\phi)C\phi + \psi=:p_1(\phi), \nonumber \\
    &&\mbox{} 0 < \phi: = p_2(\phi), \quad  0 < \frac{1}{Q} - \phi:=p_3(\phi) \quad \mbox{and} \nonumber \\ 
        &&\mbox{} 0 < 1 - \phi B(\phi) C =: p_4(\phi), \nonumber 
\eeqa
with $\psi = \frac{\nu}{Q^4}$. It is straightforward to check that these four inequalities $\left\{0<p_i(\phi)\right\}_{i=1}^4$ are always satisfied for $\phi$ positive and around zero. Furthermore, they are all polynomial inequalities as $B(\phi)$ is a polynomial in $\phi$. Thus, by a standard root-finding algorithm, we can compute all the roots of these four polynomial equations $\left\{0=p_i(\phi)\right\}_{i=1}^4$ accurately and set $\overline{\phi}_\nu$ to the smallest positive root among all the roots. With this choice of $\overline{\phi}_\nu$, the inequalities \eqref{ineq-polynomial-cond-for-lin-conv}--\eqref{ineq-polynomial-cond-2} are satisfied for $0<\phi<\overline{\phi}_\nu$ leading to local linear convergence.
\end{remark}
\section{Linear convergence with constant stepsize}
\label{sec-lin-conv-constant-step}
Consider the IN iteration \eqref{eq-iter-newton-via-error}. In this section, we will analyze the case when $\gamma^k$ is equal to a constant $\gamma$ without requiring the variable stepsize rule (Assumption \ref{assum-stepsize}) to hold. We start with two lemmas that provide bounds on the norm of the difference of gradients at inner iterates 
$x_j^k$ and $x_1^k$, and also on the overall gradient error defined in \eqref{eq-grad-error-alg2}. Note that both these lemmas hold for arbitrary (normalized) stepsizes $\gamma^k$. The first lemma is inspired by Solodov \cite{Solodov98IncrGrad}. 

\begin{lemma}\label{lemm-bound-delta-jk} Suppose that Assumptions \ref{assum-compact} and \ref{assum-strong-cvx} hold. Let $\{x_1^k,x_2^k,\dots,x_m^k\}_{k=1}^\infty$ be the IN iterates generated by \eqref{eq-inner-update}--\eqref{eq-hessian-update-convention}. For any given $k\geq 1$, let 
 \beq\label{def-delta-jk} \delta_j^k := \| \nabla f_j(x_j^k) - \nabla f_j(x_1^k) \|, \quad j=1,2,\dots,m. 
 \eeq
Then, 
  \beq\label{ineq-grad-error} \delta_j^k \leq r^k \sum_{i=1}^{j-1} (1+r^k)^{j-1-i} \| \nabla f_i(x_1^k)\| \quad \mbox{for all} \quad k\geq 2,
  \eeq
with the convention that the right-hand side of \eqref{ineq-grad-error} is zero for $j=1$, where 
    \beq\label{def-rk} r^k = \frac{2Q}{m}\gamma^k
    \eeq  
and $\gamma^k$ is the normalized stepsize defined in \eqref{eq-normalized-step}.
\end{lemma}
\begin{proof}
   Let $k\geq 2$ be given. When $j=1$, $\delta_1^k=0$ so the statement is clearly true. For $j=2$, 
\beqas
	\delta_2^k &\leq & \hessub \| x_2^k - x_1^k \| 
	=  \hessub \alpha^k \big\| D_1^k \nabla f_1(x_1^k) \big\| 
	\leq  r^k  \big\| \nabla f_1(x_1^k)\big\|
\eeqas  
where we used \eqref{eq:gradient-Lipschitzness} on the Lipschitzness of the gradients in the (first) inequality, the representation of inner iterates by the formula \eqref{eq-iter-ver-compact2} in the first equality and Lemma \ref{eq-iter-hess-bound} to bound $D_1^k$ in the second inequality. Thus, the statement is also true for $j=2$. We will proceed with an induction argument. Suppose \eqref{ineq-grad-error} is true for $j=1,2,\dots,\ell$ with $\ell<m$. Then, we have the upper bounds
\beqa \| \nabla f_j (x_j^k)\| &\leq & \| \nabla f_j(x_1^k)\| + \delta_j^k \nonumber \\ 
&\leq &
\| \nabla f_j(x_1^k)\| + r^k \sum_{i=1}^{j-1} (1+r^k)^{j-1-i} \| \nabla f_i(x_1^k)\|, 
\label{ineq-grad-at-inner-iter-bound}
\eeqa
for $j=1,2,\dots,\ell$. We will show that \eqref{ineq-grad-error} is also true for $j=\ell+1$. Using similar estimates as before,
\beqa\label{ineq-distance-bound-inner-iters-start}
  \delta_{\ell+1}^k &\leq & \hessub \| x_{\ell+1}^k - x_1^k \| \\
  &\leq&  \hessub \sum_{i=1}^l \| x_{i+1}^k - x_i^k \| 
   =  \hessub \sum_{j=1}^l \alpha^k \| D_j^k \nabla f_j(x_j^k) \| \nonumber \\
   &\leq&  r^k \sum_{j=1}^\ell \| \nabla f_j(x_j^k) \| \nonumber \\
  & \leq & r^k \sum_{j=1}^\ell \bigg( \|\nabla f_j(x_1^k)\| + r^k \sum_{i=1}^{j-1} (1+r^k)^{j-1-i} \| \nabla f_i(x_1^k)\| \bigg) \nonumber \\
  & = & r^k \sum_{j=1}^\ell (1 + r^k)^{\ell - j} \| \nabla f_j (x_1^k) \| \label{ineq-distance-bound-inner-iters-end}
\eeqa 
where we used \eqref{eq:gradient-Lipschitzness} on the Lipschitzness of the gradients in the first inequality, Lemma \ref{eq-iter-hess-bound} to bound $D_j^k$ terms in the third inequality and \eqref{ineq-grad-at-inner-iter-bound} to bound gradients in the fourth inequality. Thus, the inequality \eqref{ineq-grad-error} is also true for $j=\ell+1$. This completes the proof. \qed
\end{proof}

The next result gives an upper bound on the norm of the gradient errors. 
\begin{lemma}\label{lemm-grad-error-inc-newton} Suppose that Assumptions \ref{assum-compact} and \ref{assum-strong-cvx} hold. The gradient error $e^k$ defined by \eqref{eq-grad-error-alg2} satisfies
 \beq\label{ineq-grad-error-bound-with-gradients} \| e^k \| \leq \bigg( r^k + \frac{2Q}{km}\bigg) \sum_{j=2}^m  \sum_{i=1}^{j-1} (1 + r^k)^{j-1 - i} \| \nabla f_i (x_1^k)\| \quad \mbox{for all} \quad k\geq 2,
 \eeq
where $r^k$ is defined by \eqref{def-rk}.
\end{lemma}
\begin{proof} Using triangle inequality and the upper bound \eqref{eq-hessian-bounds} on the Hessian, the gradient error given by \eqref{eq-grad-error-alg2} admits the bound
\beq\label{eq-grad-error-delta-jk-bound}
	\| e^k \|  \leq \sum_{j=2}^m \bigg( \delta_j^k + \frac{1}{k \gamma^k} \hessub \| x_j^k - x_1^k\| \bigg) \\
\eeq
where $\gamma^k$ is the normalized stepsize defined by \eqref{eq-normalized-step}. By the estimates \eqref{ineq-distance-bound-inner-iters-start}--\eqref{ineq-distance-bound-inner-iters-end}, we have also
  \beq\label{ineq-delta-l-k} \delta_{l+1}^k \leq \hessub \| x_{\ell+1}^k - x_1^k \| \leq r^k \sum_{j=1}^\ell (1 + r^k)^{\ell - j} \| \nabla f_j (x_1^k)\| 
  \eeq
for any $\ell=1,2,\dots,m-1$. Combining \eqref{eq-grad-error-delta-jk-bound} and \eqref{ineq-delta-l-k},
\beqas
 \| e^k \| &\leq &(1 + \frac{1}{k\gamma^k} )r^k \sum_{j=2}^m  \sum_{i=1}^{j-1} (1 + r^k)^{\ell - i} \| \nabla f_i (x_1^k)\| \\
 &=& \bigg( r^k + \frac{2Q}{km}\bigg) \sum_{j=2}^m  \sum_{i=1}^{j-1} (1 + r^k)^{j -1 - i} \| \nabla f_i (x_1^k)\|
\eeqas
as desired.   
\qed
\end{proof}

Lemma \ref{lemm-grad-error-inc-newton} shows that under Assumptions \ref{assum-compact} and \ref{assum-strong-cvx} on the boundedness of the iterates, gradients and Hessian matrices, the gradient error $e^k$ is bounded as long as $\gamma^k$ is bounded and it goes to zero as $k\to \infty$ if $\gamma^k\to 0$. Then, by \cite[Theorem 1, Section 4.2.2]{Pol87}, the iterates $\{x_1^k\}_{k=1}^\infty$ generated by \eqref{eq-iter-newton-via-error}
with a constant stepsize $\gamma^k = \gamma$ converge to an $\varepsilon$-neighborhood of the optimal solution linearly for some $\varepsilon>0$ and the size of the neighborhood shrinks down as the stepsize is decreased, i.e. $\varepsilon \to 0$ as $\gamma \to 0$. This type of result was also achieved for the subgradient method \cite{NedicBert2001IncSubgrad} and the incremental gradient method for least square problems \cite[Section 1.5.2]{Bertsekas99nonlinear}. In this paper, our focus will be the optimal solution rather than approximate solutions.

The next theorem shows that under Assumption \ref{assum-gradfi} on the gradient growth,
we have global linear convergence with a constant stepsize rule if the stepsize is small enough. This is stronger than the local linear convergence obtained for the variable stepsize rule, however unlike the variable stepsize rule, if Assumption \ref{assum-gradfi} does not hold, the constant stepsize rule does not guarantee convergence to the optimal solution but only convergence to an $\varepsilon$-neighborhood of the optimal solution for some $\varepsilon>0$. We first prove a lemma.
\begin{lemma}\label{lemm-global-lin-conv-constant-step} Consider the IN iterates $\{x_1^k\}_{k=1}^\infty$ generated by \eqref{eq-iter-newton-via-error} with a constant stepsize $\gamma^k =\gamma$. Suppose that Assumptions \ref{assum-compact} and \ref{assum-strong-cvx} hold and there exists a positive integer $\hatk$ such that the gradient error $e^k$ defined by \eqref{eq-grad-error-alg2} satisfies
  \beq\label{bound-grad-error-relative} \|e^k\| \leq \p \| \nabla f(x_1^k) \| \quad \mbox{for all} \quad k \geq \hatk, \quad 0 \leq \p < \frac{1}{Q}. \eeq
It follows that there exists constants $\hatg>0$, $\hata > 0$ and $0<\hatr<1$ such that if $0<\gamma<\hatg$, then
   \beq\label{ineq-lin-conv-constants} \|x_1^k - x^*\| \leq \hata \hatr ^k \quad \mbox{for all} \quad k \geq \hatk,
   \eeq
where $x^*$ is the unique optimal solution of the problem \eqref{pbm-multi-agent}.
\end{lemma}
\begin{proof}
   Take $V(x) = f(x) - f(x^*)$ as a Lyapunov function, following the proof of \cite[Theorem 2, Section 4.2.3]{Pol87} closely. The iteration \eqref{eq-iter-newton-via-error} is equivalent to 
   $$ x_1^{k+1} = x_1^k - \gamma s^k, \quad s^k = \bar{D}_k(\nabla f(x_1^k) + e^k \big), \quad \bar{D}^k := (\bar{H}_k)^{-1}, $$
where \beq\label{bounds-aver-hessian-inv} \frac{1}{Cm}I \preceq \bar{D}^k \preceq \frac{1}{cm}I
\eeq
by taking the inverse of the bounds for $\bar{H}_k$ given in \eqref{eq-averaged-hessian-bounds}. Let $\langle \cdot,\cdot \rangle$ denote the Euclidean dot product on $\R^n$. We compute 
  \beqas \langle \nabla V (x_1^k), s^k \rangle &=& \langle \nabla f(x_1^k), \bar{D}_k\nabla f(x_1^k) + \bar{D}_k e^k  \rangle \\
   &\geq & \frac{1}{Cm} \| \nabla f(x_1^k) \|^2 - \frac{1}{cm} \p\| \nabla f(x_1^k) \|^2 =  \frac{1}{Cm} (1-\p Q)\| \nabla f(x_1^k) \|^2 \\
   & \geq & \frac{2}{Q}(1-\p Q) V(x_1^k) \geq 0,
  \eeqas
where we used \eqref{bounds-aver-hessian-inv} for bounding $\bar{D}_k$ from below in the first inequality and the strong convexity with constant $cm$ of $f$ implied by Assumption \ref{assum-strong-cvx} in the second inequality.
Similarly, from \eqref{eq:gradient-Lipschitzness}, it follows that the gradients of $f$ are Lipschitz with constant $Cm$, leading to, for $k \geq \hatk$,
  $$ \|s^k\|^2 \leq \| \bar{D}_k(\nabla f(x_1^k) + e^k \big)\|^2 \leq \frac{(1+\p)^2}{(cm)^2}  \|\nabla f(x_1^k)\|^2 \leq 2Cm \frac{(1+\p)^2}{(cm)^2} V(x_1^k),$$
where we used \eqref{bounds-aver-hessian-inv} to bound $\bar{D}_k$ from above together with the bound \eqref{bound-grad-error-relative} on the gradient error in the second inequality. 
Then, by \cite[Theorem 4, Section 2.2]{Pol87}, there exists constants $\hatg > 0$ and $\rho\in (0,1)$ such that for any $0 < \gamma < \hatg$, the iterations are linearly convergent after the $\hatk$--th step, satisfying
  \beqas f(x_1^k) - f(x^*) &\leq & \big(f(x_1^{\hatk}) - f(x^*)\big)\rho^{k-\hatk} \quad \mbox{for all} \quad k \geq \hatk.
\eeqas
Using the bounds \eqref{eq-hessian-bounds} on the Hessian of $f_i$, we have $cmI \preceq \nabla^2 f(x) \preceq CmI$ for all $x\in \R^n$. This implies the following strong convexity-based inequalities, for all $k \geq \hatk$,
   \beqas \frac{cm}{2}\|x_1^k - x^*\|^2 \leq f(x_1^k) - f(x^*) &\leq&  \frac{Cm}{2} \|x_1^{\hatk} - x^*\|^2 \rho^{k-\hatk} \nonumber \leq \frac{Cm}{2}R^2  \rho^{-\hatk} \rho^{k},
   \eeqas
where $R$ is the diameter of $\mathcal{X}$ defined by \eqref{def-R-diameter}. Hence, \eqref{ineq-lin-conv-constants} holds with $\hatr = {\rho}^{1/2} > 0$ and $\hata=({QR^2\rho^{-\hatk}})^{1/2}> 0$. This completes the proof. \qed         
\end{proof}
\begin{theorem} \textbf{(Linear convergence with constant stepsize)} Consider the iterates $\{x_1^k\}_{k=1}^\infty$ generated by  \eqref{eq-iter-newton-via-error} with a constant stepsize $\gamma^k =\gamma$. Suppose that Assumptions \ref{assum-compact}, \ref{assum-strong-cvx} and \ref{assum-gradfi} hold. Then, there exists a constant $\widetilde{\gamma}$ (depending on $M,\hesslb,\hessub$ and $m$) such that if $0 < \gamma < \widetilde{\gamma}$, the iterates are globally linearly convergent, i.e.,
 \beq\label{ineq-global-lin-conv-constant-step} \|x_1^k - x^*\| \leq A \rho^k, \quad \mbox{for all} \quad k=1,2,\dots,
 \eeq
for some constants $A>0$ and $\rho<1$. 
\end{theorem}
\begin{proof} Under Assumption \ref{assum-gradfi} on the gradient growth, the bound \eqref{ineq-grad-error-bound-with-gradients} implies that the gradient error admits the bound
   \beqas \| e^k \| &\leq & M \bigg( r^k + \frac{2Q}{km}\bigg) \sum_{j=2}^m  \sum_{i=1}^{j-1} (1 + r^k)^{j-1 - i} \| \nabla f (x_1^k)\| \\
    &\leq & M  \bigg( r^k + \frac{2Q}{km}\bigg)  m(1 + r^k)^{m-2} \| \nabla f (x_1^k)\|.
   \eeqas
Let $\hatk$ be the smallest positive integer greater than $12MQ^2$. Assume $\gamma^k = \gamma < \frac{1}{12MQ^2}$ so that $r^k < \frac{1}{6MmQ}$ where $r^k$ is defined by \eqref{def-rk}. Then, for $k\geq \hatk$, we have 
 \beqa\label{ineq-error-cond-const-stepsize} \| e^k \| & < & \frac{1}{3}\frac{1}{Q}(1 + r^k)^{m-2} \| \nabla f (x_1^k)\| 
      \leq  \frac{2}{3}\frac{1}{Q}\| \nabla f (x_1^k)\|,
   \eeqa
if $r^k\leq \sqrt[m]{2}-1$ or equivalently if $\gamma < \frac{m}{2Q}(\sqrt[m]{2}-1)$ by the definition of $r^k$ (see \eqref{def-rk}). 
Combining this with Lemma \ref{lemm-global-lin-conv-constant-step}, we conclude that there exists $\widetilde{\gamma}>0$ such that when $0 <\gamma < \widetilde{\gamma}$, the error bound \eqref{ineq-error-cond-const-stepsize} is satisfied for $k\geq \hatk$ and there exists $\hata>0$ and $\hatr<1$ such that
   $$\|x_1^k - x^*\| \leq \hata \hatr^k \quad \mbox{for all} \quad k \geq \hatk.$$
Then, a choice of $A= \max(\hata,R)/{\hatr^{\hatk}}$ where $R$ is as in \eqref{def-R-diameter} and $\rho=\hatr$ satisfies \eqref{ineq-global-lin-conv-constant-step} which completes the proof. \qed
\end{proof}

\section{An example with sublinear convergence}
\label{sec-examples-sublinear-conv}

In the following simple example, we show that 
the normalized stepsize $\gamma^k$ has to go to zero if Assumption \ref{assum-gradfi} on the gradient growth does not hold, illustrating the sublinear convergence behavior that might arise without Assumption \ref{assum-gradfi}.

\begin{example}\label{example-sublinear-conv} Let $f_1 = 1000 x + \varepsilon x^2$ and $f_2 = -1000x + \varepsilon x^2$ for a fixed $\varepsilon>0$ with $m=2$ and $n=1$. This leads to a quadratic function $f=2\varepsilon x^2$ with a unique optimal solution, $x^* = 0$ and condition number $Q=1$. We have
\beqa\label{eq-example-gradient-hessian-1} \nabla f_1(x) = 1000 + 2\varepsilon x &,& \quad \nabla f_2(x) = -1000 + 2\varepsilon x, \\ 
  \nabla^2 f_1(x)=2\varepsilon , \quad \nabla^2 
   f_2(x)=2\varepsilon&,& \quad 
   H_1^k = 2\varepsilon (2k-1) , \quad H_2^k = 4\varepsilon k \label{eq-example-gradient-hessian-3}.
\eeqa  
Assumption \ref{assum-gradfi} is clearly not satisfied, as the gradients of $f_1$ and $f_2$ do not vanish at the optimal solution 0. Rewriting the IN iterations as an inexact perturbed Newton method as in \eqref{eq-iter-newton-via-error}, we find that
  \beq\label{eq-iter-example-perturbed-form} x_1^{k+1} = x_1^k - \alpha^k \frac{1}{4\varepsilon k}(\nabla f(x_1^k) + e^k ) \eeq
with the gradient error $e^k$ given by the formula \eqref{eq-grad-error-alg2} reducing to 
 \beqa e^k &=&  \nabla f_2 (x_2^k)-\nabla f_2 (x_1^k) + \frac{1}{\alpha^k} \nabla^2 f_2(x_2^k)(x_1^k - x_2^k) \nonumber \\
        &=& -\frac{\alpha^k - 1}{2k - 1} \nabla f_1(x_1^k) 
        = -\bigg(  \frac{\gamma^k }{2} - \frac{1}{2k}\bigg)\frac{2k}{2k-1}  \nabla f_1(x_1^k) \label{eq-grad-error-final-estimate}
 \eeqa
where we used formulas \eqref{eq-example-gradient-hessian-1}--\eqref{eq-example-gradient-hessian-3}, the inner update equation \eqref{eq-inner-update} and the definition of the normalized stepsize $\gamma^k$ in \eqref{eq-normalized-step}.

For global convergence, a necessary condition is to have gradient error $e^k \to 0$. From \eqref{eq-grad-error-final-estimate} and the fact that $\nabla f_1(x_1^k)$ is bounded away from zero around the optimal solution 0, we see that this requires $\gamma^k \to 0$. Hence, we assume $\gamma^k \to 0$. In the special case, if $\alpha^k = 1$ for some $k$, then $e^k = 0$ and the IN iterations \eqref{eq-iter-example-perturbed-form} converges in one cycle, as the quadratic approximations $\tilde{f}_j$ to $f_j$ defined by \eqref{eq-quadratic-approx-to-f} become exact. Assume otherwise that $\alpha^k > 1$ for any $k$, we will show sublinear convergence by a simple classical analysis (the case $\alpha^k < 1$ for all $k$ can be handled similarly).  
Combining \eqref{eq-iter-example-perturbed-form} and \eqref{eq-grad-error-final-estimate} and plugging in the formula \eqref{eq-example-gradient-hessian-1} for the gradient of $f_1$, we can express the IN iteration as
  \beqa x_1^{k+1} 
  &=& \bigg(1 - \frac{\alpha^k}{2k}\bigg) \bigg(1 - \frac{\alpha^k}{2k-1} \bigg) x_1^k + \frac{1000}{2\varepsilon}  \frac{\alpha^k}{ 2k}\frac{\alpha^k - 1}{2k - 1}. \label{example-incr-newt-iter-quadratic} 
  \eeqa
As $\gamma^k = \alpha^k/k \to 0$, there exists a positive integer $\hat k$ such that 
  \beq\label{eq-bounds-on-mk} 1 \geq  m_k := \bigg(1 - \frac{\alpha^k}{2k}\bigg) \bigg(1 - \frac{\alpha^k}{2k-1} \bigg)> 0, \quad \mbox{for all} \quad k\geq \hat k.
  \eeq
Then, from \eqref{example-incr-newt-iter-quadratic} and \eqref{eq-bounds-on-mk}, for $x_1^{\hat k}>0$ and $k \geq \hat k$, we have the lower bounds
\beqa\label{ineq-lower-bd-example-decay} x_1^{k+1} &>& \big(\prod_{j=\bar k}^{k} m_j\big)x_1^{\hat k} > 0, \quad
  x_1^{k+1} > \frac{1000}{2\varepsilon}  \frac{\alpha^k}{ 2k}\frac{\alpha^k - 1}{2k - 1}>0. 
\eeqa
For global convergence, by \eqref{ineq-lower-bd-example-decay}, we need $\prod_{k=1}^\infty m_k = 0$ because otherwise we would have $\limsup_{k\to \infty} \big(\prod_{j=1}^k m_j\big) > \delta$ for some $\delta>0$ and 
any $x_1^{\bar k}$ satisfying $x_1^{\hat k} \geq 1/\delta$ would lead to
 $\limsup_{k\to\infty} x_1^k \geq 1$ which would be a contradiction with global convergence.  
Note that
   \beqas \prod_k m_k = 0 &\iff & \sum_k -\log(m_k) = \infty 
    \iff  \sum_k \gamma^k = \infty.
    \eeqas
where we used $2z \geq -\log(1-z)\geq z$ for $z\geq 0$ around zero and the definition \eqref{eq-normalized-step} of $\gamma^k$. Thus, the sequence $\{\gamma^k\}$, having an infinite sum, cannot decay faster than $1/k^{1+\mu}$ for any $\mu > 0$ and by the lower bound \eqref{ineq-lower-bd-example-decay}, 
convergence to the optimal solution 0 cannot be faster than $O\big(1/k^{2(1+\mu)}\big)$ for any $\mu>0$ and is thus sublinear.
  
\end{example}
\section{Conclusion}\label{sec-discussion-and-future}
We developed and analyzed an incremental version of the Newton method, proving its global convergence with alternative variable stepsize rules under some assumptions.

Furthermore, under a gradient growth assumption, we show that it can achieve linear convergence both under a constant stepsize and a variable stepsize. A by-product of our analysis is the linear convergence of the EKF-S method of \cite{AlgEkfs2003} under similar assumptions. 
Our results admit straightforward extensions to incremental quasi-Newton methods and shed light into their convergence properties as well.

\bibliographystyle{plain}
\bibliography{ismp2015_refs}   

\begin{thebibliography}{10}

\bibitem{Bertsekas1996incremental}
D.~Bertsekas.
\newblock Incremental least squares methods and the extended {K}alman filter.
\newblock {\em SIAM Journal on Optimization}, 6(3):807--822, 1996.

\bibitem{BersekasIncrGrad97}
D.~Bertsekas.
\newblock A new class of incremental gradient methods for least squares
  problems.
\newblock {\em SIAM Journal on Optimization}, 7(4):913--926, 1997.

\bibitem{Bertsekas99nonlinear}
D.~Bertsekas.
\newblock {\em Nonlinear programming}.
\newblock Athena Scientific, 1999.

\bibitem{bertsekas2011incremental}
D.~Bertsekas.
\newblock Incremental gradient, subgradient, and proximal methods for convex
  optimization: a survey.
\newblock {\em Optimization for Machine Learning}, 2010:1--38, 2011.

\bibitem{Blatt2007incremental}
D.~Blatt, A.~Hero, and H.~Gauchman.
\newblock A convergent incremental gradient method with a constant step size.
\newblock {\em SIAM Journal on Optimization}, 18(1):29--51, 2007.

\bibitem{Bordes2009sgdqn}
A.~Bordes, L.~Bottou, and P.~Gallinari.
\newblock {SGD-QN}: Careful quasi-{N}ewton stochastic gradient descent.
\newblock {\em The Journal of Machine Learning Research}, 10:1737--1754, 2009.

\bibitem{Borges2010BookChapter}
L.~Bottou.
\newblock Large-scale machine learning with stochastic gradient descent.
\newblock In Y.~Lechevallier and G.~Saporta, editors, {\em Proceedings of
  COMPSTAT'2010}, pages 177--186. Physica-Verlag HD, 2010.

\bibitem{BottouLecun2005}
L.~Bottou and Y.~Le~Cun.
\newblock On-line learning for very large data sets.
\newblock {\em Applied Stochastic Models in Business and Industry},
  21(2):137--151, 2005.

\bibitem{Boyd2011AdmmBook}
S.~Boyd, N.~Parikh, E.~Chu, B.~Peleato, and J.~Eckstein.
\newblock Distributed optimization and statistical learning via the alternating
  direction method of multipliers.
\newblock {\em Foundations and Trends® in Machine Learning}, 3(1):1--122,
  2011.

\bibitem{StocBFGSNocedal14}
R.~H. {Byrd}, S.~L. {Hansen}, J.~{Nocedal}, and Y.~{Singer}.
\newblock {A Stochastic Quasi-Newton Method for Large-Scale Optimization}.
\newblock {\em arXiv preprint arXiv:1401.7020}, 2014.

\bibitem{Catinas2001Jota}
E.~C\u{a}tina\c{s}.
\newblock Inexact perturbed {N}ewton methods and applications to a class of
  {K}rylov solvers.
\newblock {\em Journal of Optimization Theory and Applications},
  108(3):543--570, 2001.

\bibitem{Davidon76}
W.C. Davidon.
\newblock New least-square algorithms.
\newblock {\em Journal of Optimization Theory and Applications},
  18(2):187--197, 1976.

\bibitem{BachSagaMethod14}
A.~Defazio, F.~Bach, and S.~Lacoste-Julien.
\newblock {SAGA}:{A} fast incremental gradient method with support for
  non-strongly convex composite objectives.
\newblock {\em arXiv preprint arXiv:1407.0202}, 2014.

\bibitem{DemboEisenNewton82}
R.~Dembo, S.~Eisenstat, and T.~Steihaug.
\newblock Inexact {N}ewton methods.
\newblock {\em SIAM Journal on Numerical Analysis}, 19(2):400--408, 1982.

\bibitem{DennisMore74}
J.~E. Dennis and J.~J. Mor\'{e}.
\newblock A characterization of superlinear convergence and its application to
  quasi-newton methods.
\newblock {\em Mathematics of Computation}, 28(126):549--560, 1974.

\bibitem{Duchi2011AdaGrad}
J.~Duchi, E.~Hazan, and Y.~Singer.
\newblock Adaptive subgradient methods for online learning and stochastic
  optimization.
\newblock {\em The Journal of Machine Learning Research}, 12:2121--2159, 2011.

\bibitem{MairalSurrogate2013}
Julien Mairal.
\newblock {Optimization with First-Order Surrogate Functions}.
\newblock In {\em {ICML}}, volume~28 of {\em JMLR Proceedings}, pages 783--791,
  Atlanta, United States, 2013.

\bibitem{ManSol98IncGradMomentum}
O.L. Mangasarian and M.V. Solodov.
\newblock Serial and parallel backpropagation convergence via nonmonotone
  perturbed minimization.
\newblock {\em Optimization Methods and Software}, 4(2):103--116, 1994.

\bibitem{StocBFGSRibeiro14}
A.~Mokhtari and A.~Ribeiro.
\newblock Res: Regularized {S}tochastic {BFGS} algorithm.
\newblock {\em arXiv preprint arXiv:1401.7625}, 2014.

\bibitem{AlgEkfs2003}
H.~Moriyama, N.~Yamashita, and M.~Fukushima.
\newblock The incremental {G}auss-{N}ewton algorithm with adaptive stepsize
  rule.
\newblock {\em Computational Optimization and Applications}, 26(2):107--141,
  2003.

\bibitem{NedicBert2001IncSubgrad}
A.~Nedi\'{c} and D.~Bertsekas.
\newblock Convergence rate of incremental subgradient algorithms.
\newblock In S.~Uryasev and P.M. Pardalos, editors, {\em Stochastic
  Optimization: Algorithms and Applications}, volume~54 of {\em Applied
  Optimization}, pages 223--264. Springer US, 2001.

\bibitem{Nedic2007rate}
A.~Nedi\'{c} and A.~Ozdaglar.
\newblock On the rate of convergence of distributed subgradient methods for
  multi-agent optimization.
\newblock In {\em Proceedings of IEEE CDC}, pages 4711--4716, 2007.

\bibitem{NedicOzdaglar09}
A.~Nedi\'{c} and A.~Ozdaglar.
\newblock Distributed subgradient methods for multi-agent optimization.
\newblock {\em IEEE Transactions on Automatic Control}, 54(1):48--61, 2009.

\bibitem{Pol87}
E.~Polak.
\newblock On the mathematical foundations of nondifferentiable optimization in
  engineering design.
\newblock {\em SIAM Review}, 29:21--90, 1987.

\bibitem{RamNedicVeer2007}
S.S. Ram, A.~Nedic, and V.V. Veeravalli.
\newblock Stochastic incremental gradient descent for estimation in sensor
  networks.
\newblock In {\em Signals, Systems and Computers, 2007. ACSSC 2007. Conference
  Record of the Forty-First Asilomar Conference on}, pages 582--586, 2007.

\bibitem{RobbinsMonro}
H.~Robbins and S.~Monro.
\newblock A stochastic approximation method.
\newblock {\em The Annals of Mathematical Statistics}, 22(3):400--407, 1951.

\bibitem{Leroux2012sgd}
N.~L. Roux, M.~Schmidt, and F.R. Bach.
\newblock A stochastic gradient method with an exponential convergence rate for
  finite training sets.
\newblock In F.~Pereira, C.J.C. Burges, L.~Bottou, and K.Q. Weinberger,
  editors, {\em Advances in Neural Information Processing Systems 25}, pages
  2663--2671. Curran Associates, Inc., 2012.

\bibitem{schmidt2013fast}
M.~Schmidt and N.L. Roux.
\newblock Fast convergence of stochastic gradient descent under a strong growth
  condition.
\newblock {\em arXiv preprint arXiv:1308.6370}, 2013.

\bibitem{Schraudolph2007stochastic}
N.~Schraudolph, J.~Yu, and S.~G{\"u}nter.
\newblock A stochastic quasi-{N}ewton method for online convex optimization.
\newblock In {\em Proceedings of the 11th International Conference Artificial
  Intelligence and Statistics (AISTATS)}, pages 433--440, 2007.

\bibitem{SrebroDane2013}
O.~Shamir, N.~Srebro, and T.~Zhang.
\newblock Communication efficient distributed optimization using an approximate
  {N}ewton-type method.
\newblock {\em ICML}, 32(1):1000--1008, 2014.

\bibitem{Dickenstein2014}
J.~Sohl-Dickstein, B.~Poole, and S.~Ganguli.
\newblock Fast large-scale optimization by unifying stochastic gradient and
  quasi-{N}ewton methods.
\newblock In T.~Jebara and E.~P. Xing, editors, {\em ICML}, pages 604--612.
  JMLR Workshop and Conference Proceedings, 2014.

\bibitem{Solodov98IncrGrad}
M.V. Solodov.
\newblock Incremental gradient algorithms with stepsizes bounded away from
  zero.
\newblock {\em Computational Optimization and Applications}, 11(1):23--35,
  1998.

\bibitem{Jordan13DistLearningApi}
E.R. Sparks, A.~Talwalkar, V.~Smith, J.~Kottalam, P.~Xinghao, J.~Gonzalez, M.J.
  Franklin, M.I Jordan, and T.~Kraska.
\newblock {MLI}: An {API} for distributed machine learning.
\newblock In {\em IEEE 13th International Conference on Data Mining (ICDM)},
  pages 1187--1192, 2013.

\bibitem{TsengIncrGradient98}
P.~Tseng.
\newblock An incremental gradient(-projection) method with momentum term and
  adaptive stepsize rule.
\newblock {\em SIAM Journal on Optimization}, 8(2):506--531, 1998.

\bibitem{TsengYun2014incremental}
P.~Tseng and S.~Yun.
\newblock Incrementally updated gradient methods for constrained and
  regularized optimization.
\newblock {\em Journal of Optimization Theory and Applications},
  160(3):832--853, 2014.

\bibitem{Zhang2004}
T.~Zhang.
\newblock Solving large scale linear prediction problems using stochastic
  gradient descent algorithms.
\newblock In {\em Proceedings of the Twenty-first International Conference on
  Machine Learning}, ICML, pages 116--, New York, NY, USA, 2004. ACM.

\end{thebibliography}


\end{document}